\documentclass[sn-mathphys]{sn-jnl}

\usepackage{lipsum}
\usepackage{amsmath}
\usepackage{amsthm}
\usepackage{amsfonts}
\usepackage{graphicx}
\usepackage{algorithm}
\usepackage{algpseudocode}
\usepackage{verbatim}
\usepackage{caption}
\usepackage{subcaption}
\usepackage{float}
\usepackage{booktabs}
\usepackage{array}
\usepackage{lineno,hyperref}
\usepackage[nameinlink,noabbrev]{cleveref}
\usepackage{amsopn}
\usepackage{calc}
\usepackage{bbm}
\usepackage{multirow}
\numberwithin{equation}{section}

\DeclareMathOperator{\diag}{diag}

\newcommand{\ang}[1]{\left\langle {#1} \right\rangle}
\newcommand {\average}[1] {\mbox{$\left\{\!\!\left\{ #1 \right\}\!\!\right\}$}}
\newcommand {\jump}[1] {\mbox{$\left[\!\left[ #1 \right]\!\right]$}}

\newtheorem{theorem}{Theorem}[section]

\newtheorem{lemma}[theorem]{Lemma}

\makeatletter

\newcommand{\DESCRIPTION@original@item}{}
\let\DESCRIPTION@original@item\item
\newcommand*{\DESCRIPTION@envir}{DESCRIPTION}
\newlength{\DESCRIPTION@totalleftmargin}
\newlength{\DESCRIPTION@linewidth}
\newcommand{\DESCRIPTION@makelabel}[1]{\llap{#1}}%
\newcommand{\DESCRIPTION@item}[1][]{%
	\setlength{\@totalleftmargin}%
	{\DESCRIPTION@totalleftmargin+\widthof{\textbf{#1 }}-\leftmargin}%
	\setlength{\linewidth}
	{\DESCRIPTION@linewidth-\widthof{\textbf{#1 }}+\leftmargin}%
	\par\parshape \@ne \@totalleftmargin \linewidth
	\DESCRIPTION@original@item[\textbf{#1}]%
}

\jyear{2022}%

\begin{document}

\title[Stochastic Discontinuous Galerkin for Robust Deterministic Control]{Stochastic Discontinuous Galerkin Methods for Robust Deterministic Control of Convection Diffusion Equations with Uncertain Coefficients}


\author[1]{\fnm{Pelin} \sur{\c{C}\.{i}lo\u{g}lu}}\email{pciloglu@metu.edu.tr}
\author*[1]{\fnm{Hamdullah} \sur{ Y\"ucel}}\email{yucelh@metu.edu.tr}

\affil[1]{\orgdiv{Institute of Applied Mathematics}, \orgname{Middle East Technical University}, \orgaddress{ \city{Ankara}, \postcode{06800},  \country{Turkey}}}


\abstract{We investigate a numerical behaviour of robust deterministic optimal control problem subject to a convection diffusion equation containing  uncertain inputs. Stochastic Galerkin approach, turning the original optimization problem containing uncertainties into a large system of deterministic problems, is applied to discretize the stochastic domain, while a discontinuous Galerkin method is preferred  for the spatial discretization due to its better convergence behaviour for optimization problems governed by convection dominated PDEs. Error analysis is done for the state and adjoint variables in the energy norm, while the estimates of deterministic control is obtained in the $L^2$--norm. Large matrix system emerging from the stochastic Galerkin method is addressed by the low--rank version of GMRES method, which reduces both the computational complexity and the memory requirements by employing  Kronecker--product structure of the obtained  linear system. Benchmark examples with and without control constraints are presented to illustrate the efficiency of the proposed methodology.}




\maketitle

\section{Introduction}\label{sec:intro}

In many phenomena in physics or engineering applications, certain parameters of a model are optimized in order to reach the desired target, for instance,  the location where the oil is inserted into the medium, the temperature of a melting/heating process, or the shape of the aircraft wings.  Such real-world phenomena can  be modelled as optimal control problems or optimization problems with PDE constraints. However, in reality, the input parameters of these simulations, such as the wind speed or material properties, are not often known due to the missing information or inherent variability in the problem; see, e.g., \cite{PJRoache_1988}.  Therefore, in the last decade, the idea of uncertainty quantification, i.e., quantifying the effects of uncertainty on the result of a computation, has become a powerful tool for modeling physical phenomena in the scientific community.

PDE--constraint optimization problems with uncertainty have been studied in various formulations in the literature, such as mean--based control \cite{ABorzi_2010a,ABorzi_VSchulz_CSchillings_GvonWinckel_2010a}, pathwise control \cite{FNegri_AManzoni_GRozza_2015,AAAli_EUllman_MHinze_2017}, average control \cite{MLazar_EZuazua_2014,EZuazua_2014}, robust deterministic control \cite{SGarreis_MUlbrich_2017,MDGunzburger_HCLee_JLee_2011a,LSHou_JLee_HManouzi_2011,DPKouri_MHeinkenschloss_DRidzal_BGWaanders_2013,HCLee_JLee_2013,ERosseel_GNWells_2012}, and   robust stochastic control \cite{PBenner_AOnwunto_MStoll_2016,PChen_AQuarteroni_GRozza_2013,AKunoth_CSchwab_2016,HTiesler_RMKirby_DXiu_TPreusser_2012a}. Robust deterministic control is more practical and realistic since randomness cannot be observed during the design of the control. Therefore,  we are here  interested with the following  robust deterministic control problem
\begin{eqnarray}\label{eqn:objec}
	\min \limits_{u \in \mathcal{U}^{ad}} \; \mathcal{J}(y,u) := \frac{1}{2} \lVert y-y^d\rVert^2_{\mathcal{X}} + \frac{\gamma}{2} \lVert \text{std}(y) \rVert^2_{\mathcal{W}} + \frac{\mu}{2} \lVert u \rVert^2_{\mathcal{U}}
\end{eqnarray}
governed by
\begin{subequations}\label{eqn:moperator}
	\begin{align}
		\mathcal{S}\big(y(\boldsymbol{x},\omega)\big)  & =  f(\boldsymbol{x}) + u(\boldsymbol{x})  \quad \; \hbox{in} \;\; \mathcal{D} \times \Omega, \\
		y(\boldsymbol{x},\omega) & =  y_{DB}(\boldsymbol{x}) \qquad  \quad \; \hbox{on} \; \;  \partial \mathcal{D} \times \Omega,
	\end{align}
\end{subequations}
where $ \mathcal{S} : \mathcal{Y} \rightarrow  \mathcal{Y}' $ is a linear operator that contains uncertain parameters, $\mathcal{D} \subset \mathbb{R}^2$ is a convex bounded polygonal set with a Lipschitz boundary $\partial \mathcal{D}$, and $\Omega$ is a sample space of events. The cost functional including a risk penalization via the standard deviation $\text{std}(y)$ is denoted by $\mathcal{J}(y,u)$. The first term in \eqref{eqn:objec} is a measure of the distance between the state variable $y$ and the desired state $y^d$ in terms of expectation of $y-y^d$. Without loss of generality, we assume that the state $y \in \mathcal{Y}$ is a random field, whereas the desired state $y^d \in \mathcal{Y}$ is modelled deterministically. The second term  measures the standard deviation of $y$, which is added since it is desirable to have a control for which the state is more accurately known, leading to a risk averse optimum.  The last term corresponds to distributive deterministic control. The constant $\mu >0 $ is a positive regularization parameter  of the control $u$, whereas $\gamma \geq 0$ is a risk--aversion parameter. Deterministic source function and Dirichlet boundary conditions are denoted by $f$ and $y_{DB}$, respectively. We note that the cost functional $\mathcal{J}$ is a deterministic quantity although it contains uncertain inputs. Further, the closed convex admissible set in the control space $\mathcal{U}$ is defined by
\begin{equation}\label{defn:add}
	\mathcal{U}^{ad} :=  \{ u \in \mathcal{U} : \; u_a \leq  u(\boldsymbol{x}) \leq u_b, \;\; \forall \boldsymbol{x} \in \mathcal{D} \},
\end{equation}
where constants $u_a, u_b \in \mathbb{R}$ with $u_a \leq u_b$.

Finding an approximate solution for the optimization problems containing uncertainty \eqref{eqn:objec}--\eqref{eqn:moperator} is extremely challenging and requires much more computational resources than the ones in the deterministic setting. In the literature, there exist various competing methods to solve such kinds of problem, for instance,  Monte Carlo (MC)  \cite{AAAli_EUllman_MHinze_2017,PAGuth_VKaarnioja_FYKuo_CSchillings_IHSloan_2021a,AVBarel_SVandewalle_2019}, stochastic collocation method (SCM) \cite{ABorzi_GvonWinckel_2009a,LGe_TSun_2019a,HTiesler_RMKirby_DXiu_TPreusser_2012a,ERosseel_GNWells_2012}, and stochastic Galerkin method (SGM) \cite{MDGunzburger_HCLee_JLee_2011a,HCLee_JLee_2013,ERosseel_GNWells_2012,TSun_WShen_BGong_WLiu_2016a,DDurrwachter_TKuhn_FMeyer_LSchacchter_FSchneider_2020}. Although the MC method is popular for its simplicity,  natural parallelization, and broad applications, it features slow convergence, which does not depend on the number of uncertain parameters \cite{GSFishman_1996,JSLiu_2008a}. For the SCMs, the crucial issue  is how to construct the set of collocation points appropriately because the choice of the collocation points determines the efficiency of the method. In contrast to the MC approach and the SCM, the SGM is a nonsampling technique, which transforms the problem into a large system of deterministic problems.  As in the classic (deterministic) Galerkin method, the idea behind the SGM is to seek a solution for the model equation such that the residue is orthogonal to the space of polynomials. Since the random process is expressed as an expansion with the help of orthogonal polynomials, the SGM is considered as  a variant  of the  generalized polynomial chaos approximation \cite{DXiu_GEKarniadakis_2002,GPoette_BDespres_DLucor_2009,POffner_JGlaubitz_HRanocha_2018}  as the stochastic collocation method. An important feature of the SGM is the separation of the spatial and stochastic variables, which  allows a reuse of established numerical techniques. The results obtained in \cite{ERosseel_GNWells_2012}  also show that the SGM generally displays superior performance compared to the SCM for the  robust deterministic control problems. Within the framework of the aforementioned features, the stochastic Galerkin method is preferred as a stochastic method in this study. On the other hand, for the discretization of the spatial domain, we use a discontinuous Galerkin method due to its better convergence behaviour for the optimization problems governed by convection dominated PDEs; see, e.g., \cite{DLeykekhman_MHeinkenschloss_2012a,HYucel_PBenner_2015a,HYucel_MHeinkenschloss_BKarasozen_2013}. We also refer to \cite{DNArnold_FBrezzi_BCockburn_LDMarini_2002a,BRiviere_2008a} and references therein for more details on the discontinuous Galerkin methods.

In spite of these nice properties exhibited by the stochastic discontinuous Galerkin method, the dimension of the resulting linear system increases rapidly, called as the curse of dimensionality. As a remedy, we apply a low--rank variant of  generalized minimal residual (GMRES) method \cite{YSaad_MHSchultz_1986} with a suitable preconditioner.  With the help of
a Kronecker--product structure of the obtained large matrices, we reduce  both the computational complexity and memory requirements; see, e.g., \cite{JBallani_LGrasedyck_2013,DKressner_CTobler_2011,MStoll_TBreiten_2015}. Low-rank approximation of the optimal control problems with uncertain terms have been also studied in \cite{PBenner_AOnwunto_MStoll_2016,PBenner_SDolgov_AOnwunta_MStoll_2016a,PBenner_SDolgov_AOnwunta_MStoll_2020a} for unconstrained control problems and in \cite{SGarreis_MUlbrich_2017a} for control constraint problems. In the aforementioned studies, randomness is generally defined on the diffusion parameter; however, we here consider the randomness  on diffusion or convection parameters by applying the discontinuous Galerkin method in the spatial domain. In addition, according to the best of our knowledge, a low--rank approximation of the optimal control problems governed by convection dominated equations containing randomness has not been discussed  before in the setting of discontinuous Galerkin discretization in the spatial domain.

We organize our paper by first discussing the existence of the solution  in the next section. In Section~\ref{sec:finite}, we reduce the problem into finite dimensional setting via Karhunen--Lo\`{e}ve (KL) expansion, stochastic Galerkin method, and symmetric interior penalty Galerkin method. Error analyses are done in Section~\ref{sec:apriori}. In Section~~\ref{sec:matrix}, we construct the matrix formulation of the underlying optimization problem by proceeding the optimize-then-discretize approach, and then
discuss  implementation of the low--rank GMRES solver. Results of the numerical experiments are provided in Section~\ref{sec:num} to illustrate  the efficiency of the proposed methodology. Finally, we end the paper with  some conclusions and discussions in Section~\ref{sec:conc}.


\section{Existence and uniqueness of the solution}\label{sec:model}

Let $\Omega$ be a sample space of events, $\mathcal{F} \subset 2^{\Omega}$ denotes a $\sigma$--algebra, and $\mathbb{P}$ is the associated probability measure that maps the events in  $\mathcal{F}$ to probabilities in $ [0,1]$.  A generic random field $\eta$ on the  probability space $(\Omega, \mathcal{F}, \mathbb{P})$ is denoted by $\eta(\boldsymbol{x},\omega):\mathcal{D} \times \Omega \rightarrow \mathbb{R}$.  For a fixed $\boldsymbol{x} \in \mathcal{D}$, $\eta(\boldsymbol{x},\cdot)$ is a real--valued square integrable random variable $\eta(\boldsymbol{x}, \cdot) \in L^2(\Omega, \mathcal{F}, \mathbb{P})$, i.e.,
\begin{equation*}
	L^2(\Omega):= L^2(\Omega, \mathcal{F}, \mathbb{P}) : = \{X: \Omega \rightarrow \mathbb{R}\, : \; \int_{\Omega} \lvert X(\omega) \rvert^2 \, d\mathbb{P}(\omega) < \infty  \}.
\end{equation*}
Then, the mean $\mathbb{E}[\eta]$, the standard deviation $\text{std}(\eta)$, and  the corresponding variance $\mathbb{V}(\eta)$  for any random field $\eta$, are given, respectively, by
\begin{align*}
	&\mathbb{E}[\eta] = \int_{\Omega}  \eta  \, d \mathbb{P}(\omega), \qquad  \qquad \quad  \text{std}(\eta) =\left[ \int_{\Omega} \left( \eta - \mathbb{E}[\eta] \right)^2 \, d \mathbb{P}(\omega) \right]^{1/2}, \\
	& \mathbb{V}(\eta) = \left[\text{std}(\eta)\right]^2 = \mathbb{E}[\eta^2] - \left(\mathbb{E}[\eta]\right)^2.
\end{align*}
Recalling the tensor--product space $H^k(\mathcal{D}) \otimes L^2(\Omega)$ equipped with the norm
\begin{eqnarray}
	\|\eta\|_{H^k(\mathcal{D}) \otimes L^2(\Omega)} := \left( \int_{\Omega} \|\eta(\cdot,\omega)\|^2_{H^k(\mathcal{D})} \, d\mathbb{P}(\omega)\right)^{1/2} < \infty,
\end{eqnarray}
the state and control spaces are defined as follow, respectively,
\[
\mathcal{Y} := H_0^1(\mathcal{D}) \otimes L^2(\Omega)   \quad \hbox{and} \quad \mathcal{U}:= L^2(\mathcal{D}).
\]
We also set $\mathcal{X}:= {L^2(\mathcal{D}) \otimes L^2(\Omega)}$ and $\mathcal{W}=L^2(\mathcal{D})$.

In order to show  existence of the solution, it is assumed  that the operator $\mathcal{S}$ satisfies the following conditions:
\begin{itemize}
	\item[a)] $\mathcal{S}$ is coercive such that $\mathbb{P}$-a.s., $(\mathcal{S}v,v)   \geq c \lVert v \rVert_{\mathcal{X}}, \;\, \forall v \in \mathcal{X}$, where $c$ is a positive constant.
	\item[b)] $(\mathcal{S}u,v) =(u, \mathcal{S}^*v) \;\; \forall u,v \in \mathcal{X}$, where $\mathcal{S}^*$ is the adjoint of $\mathcal{S}$.
\end{itemize}
By following the standard arguments in the theory of optimal control, see, e.g., \cite[Theorem 1.3]{JLLions_1971}  and \cite[Theorem 2.14]{FTroeltzsch_2010a}, the existence and uniqueness of an optimal solution for the optimization problem \eqref{eqn:objec}--\eqref{eqn:moperator} can be proved. With the definitions above, $\mathcal{Y}$ and $\mathcal{U}$ are Hilbert spaces, the functional $\mathcal{J}$ is strictly convex, and the admissible set  $\mathcal{U}^{ad}$ is a closed and convex set. Then, according to Lion's Lemma \cite[Theorem 1.3]{JLLions_1971}, a unique optimal control $ \bar{u} \in \mathcal{U}$ exists and  the  variational inequality holds
\begin{eqnarray}\label{ineq:direc}
	\mathcal{J}'(\bar{u}) \cdot (\upsilon-\bar{u}) \geq 0, \quad \forall \upsilon \in \mathcal{U}^{ad}.
\end{eqnarray}

Now, we can state  the first  order optimality system of the optimization problem containing uncertain coefficients \eqref{eqn:objec}--\eqref{eqn:moperator}.

\begin{theorem}
	A pair $(y,u)$ is a unique  solution of the optimization  problem  \eqref{eqn:objec}--\eqref{eqn:moperator} if and only if there exists an adjoint  $ p \in \mathcal{Y} $ such that the  optimality system holds, $\mathbb{P}$-a.s., for the triplet $ (y(u), u, p(u)) \in \mathcal{Y} \times \mathcal{U}^{ad} \times \mathcal{Y}$
	\begin{subequations}\label{eqn:opti}
		\begin{align}
			& \mathcal{S}\big(y(u)\big)  = f(\boldsymbol{x}) + u(\boldsymbol{x}),  \label{eqn:Ay} \\
			& \mathcal{S}^* \big(p(u)\big) = y(u)  -y^d  +  \gamma \big( y(u) - \mathbb{E}[y(u)] \big), \\
			& \big( \mathbb{E}[p(u)] + \mu u, v-u \big) \geq  0,  \qquad     v \in \mathcal{U}^{ad}.
		\end{align}
	\end{subequations}	
\end{theorem}

\begin{proof}
	Rewrite the objective functional $\mathcal{J}$ as
	\begin{eqnarray*}
		\mathcal{J}(u) & = &
		\underbrace{\frac{1}{2} \mathbb{E} \left[ \int_{\mathcal{D}} \left( y(u) -y^d\right)^2 \, d\boldsymbol{x} \right]}_{\mathcal{J}_1(u)} + \underbrace{\frac{\gamma}{2}  \mathbb{E} \left[ \int_{\mathcal{D}} y(u)^2 \, d\boldsymbol{x} \right]}_{\mathcal{J}_2(u)}  \\
		&& - \underbrace{\frac{\gamma}{2}  \int_{\mathcal{D}} \left(\mathbb{E}\left[y(u)\right]\right)^2 \, d\boldsymbol{x} }_{\mathcal{J}_3(u)} + \underbrace{\frac{\mu}{2}  \int_{\mathcal{D}} u^2 \, d\boldsymbol{x}}_{\mathcal{J}_4(u)}.
	\end{eqnarray*}
	By the definition of directional derivative, we obtain that
	\begin{align}\label{eqn:direcJ}
		&	\mathcal{J}'(u) \cdot (v-u ) \nonumber \\
		& \quad  =  \mathbb{E} \left[ \int_{\mathcal{D}} \big( y(u) -y^d \big) y'(u)\cdot (v-u) \, d\boldsymbol{x} \right] + \gamma \mathbb{E} \left[ \int_{\mathcal{D}}y(u)y'(u)\cdot (v-u) \, d\boldsymbol{x} \right] \nonumber \\
		&\qquad -   \gamma  \int_{\mathcal{D}}  \mathbb{E}[y(u)]  y'(u)\cdot (v-u) \, d\boldsymbol{x}  +  \mu  \int_{\mathcal{D}} u \cdot (v-u) \, d\boldsymbol{x}.
	\end{align}
	By well--posedness of the state equation \eqref{eqn:moperator} followed from the Lax--Milgram lemma, one can easily show that the operator $\mathcal{S}$ is invertible so that, by taking directional derivative, one  gets
	\begin{eqnarray*} \label{eqn:ydirec}
		y'(u) \cdot (v-u) =  \mathcal{S}^{-1}(v-u) =  y(v) -y(u).
	\end{eqnarray*}
	Thus,   \eqref{eqn:direcJ} gives us
	\begin{equation}\label{eqn:direcShort}
		\mathcal{J}'(u) \cdot (v-u )  =  \Psi(\gamma) + \mu  \int_{\mathcal{D}} u \cdot (v-u) \, d\boldsymbol{x},
	\end{equation}
	where
	\begin{align*}
		\Psi(\gamma) & =  (1+\gamma) \mathbb{E} \left[ \int_{\mathcal{D}} y(u) \cdot \big( y(v) -y(u) \big)\, d\boldsymbol{x} \right] - \gamma  \int_{\mathcal{D}}  \mathbb{E}[y(u)] \cdot \big( y(v) -y(u) \big) \, d\boldsymbol{x}  \\
		& \quad  -  \mathbb{E} \left[ \int_{\mathcal{D}}  y^d \cdot \big( y(v) -y(u) \big) \, d\boldsymbol{x} \right].
	\end{align*}
	To guarantee  the existence and uniqueness of the solution from Lion's Lemma \cite[Theorem 1.3]{JLLions_1971}, we need the following requirement
	\begin{equation}\label{eqn:lione}
		\mathcal{J}'(u) \cdot (v-u ) \geq 0.
	\end{equation}
	Next, we introduce  the adjoint state $p(u) \in \mathcal{Y}$ by
	\begin{equation}\label{eqn:vadjo}
		\mathcal{S}^* \big(p(u)\big)  =  y(u) - y^d + \gamma \big( y(u) - \mathbb{E}[y(u)] \big).
	\end{equation}
	Multiplying both sides of \eqref{eqn:vadjo} by $\big(y(v) - y(u)\big)$, integrating over $\mathcal{D}$, and taking  the expectation of the resulting system,  we obtain
	\begin{eqnarray}\label{eqn:vadjo2}
		\mathbb{E} \left[ \int_{\mathcal{D}} \mathcal{S}^* \big(p(u)\big) \cdot \big( y(v) - y(u) \big) \, d\boldsymbol{x} \right]
		& = & \mathbb{E} \left[ \int_{\mathcal{D}} p(u) \cdot \big( \mathcal{S}\big(y(v)) - \mathcal{S}\big(y(u)) \big) \, d\boldsymbol{x} \right] \nonumber \\
		& = & \mathbb{E} \left[ \int_{\mathcal{D}} p(u) \cdot \big(v - u  \big) \, d\boldsymbol{x} \right] \nonumber  \\
		& = & \Psi(\gamma).
	\end{eqnarray}
	Inserting \eqref{eqn:vadjo2} into \eqref{eqn:direcShort} and combining with \eqref{eqn:lione} give us
	\begin{eqnarray}
		\mathcal{J}'(u) \cdot (v-u)
		& = & \big( \mathbb{E}[p(u)] + \mu u, v-u \big) \geq 0,
	\end{eqnarray}
	which is the desired result.
\end{proof}

In this study, we consider $\mathcal{S}$  as the convection--diffusion operator
\begin{eqnarray}
	\mathcal{S} := 	-\nabla \cdot \big( a(\boldsymbol{x},\omega) \nabla \big) + \mathbf{b}(\boldsymbol{x},\omega)\cdot \nabla,
\end{eqnarray}
which turns the state equation \eqref{eqn:moperator} into
\begin{subequations}\label{eqn:m1}
	\begin{align}
		-\nabla \cdot \big( a(\boldsymbol{x},\omega) \nabla y \big) + \mathbf{b}(\boldsymbol{x},\omega)\cdot \nabla y & =  f + u  \qquad \;  \hbox{in} \; \mathcal{D} \times \Omega, \\
		y & =  y_{DB} \quad  \quad \;\;\; \hbox{on} \; \partial \mathcal{D} \times \Omega,
	\end{align}
\end{subequations}
where  $a:(\mathcal{D} \times \Omega) \rightarrow \mathbb{R}$ and $\mathbf{b}:(\mathcal{D} \times \Omega) \rightarrow \mathbb{R}^2$ are  random diffusivity and velocity coefficients, respectively, which is assumed to have continuous and bounded covariance functions. In addition, we make the following  assumptions on the uncertain coefficients:
\begin{itemize}
	\item[i)]  $\exists \; a_{min}, a_{max}$  such that for almost every  $(\boldsymbol{x},\omega) \in  \mathcal{D} \times \Omega$, $0 < a_{min} \leq  a(\boldsymbol{x},\omega) \leq a_{max} < \infty$. In addition, $a(\boldsymbol{x},\omega)$ has a uniformly bounded and continuous first derivatives.
	\item[ii)] The velocity coefficient $\mathbf{b}$ satisfies $\mathbf{b}(\cdot, \omega) \in \big( L^{\infty}(\overline{\mathcal{D}}) \big)^2$  for  a.e.  $\omega \in \Omega$ and $\nabla \cdot \mathbf{b}(\boldsymbol{x},\omega) =0$.
\end{itemize}
Then, the well--posedness of the state equation \eqref{eqn:m1} can be shown by following  the classical Lax--Milgram lemma; see, e.g., \cite{IBabuska_RTempone_GEZouraris_2004a,GJLord_CEPowell_TShardlow_2014}.

Now, we give the corresponding weak formulation of the optimization problem containing uncertainty \eqref{eqn:objec}--\eqref{eqn:moperator} as follows
\begin{eqnarray}\label{eqn:objec2}
	\min \limits_{u \in \mathcal{U}^{ad}} \; \mathcal{J}(u) &=& \frac{1}{2} \mathbb{E} \left[ \int_{\mathcal{D}} \left( y(u) -y^d\right)^2 \, d\boldsymbol{x} \right] + \frac{\gamma}{2}   \mathbb{E} \left[ \int_{\mathcal{D}} \big(y(u) - \mathbb{E}[y(u)]\big)^2 \, d\boldsymbol{x} \right] \nonumber  \\
	&&+ \frac{\mu}{2}  \int_{\mathcal{D}} u^2 \, d\boldsymbol{x}
\end{eqnarray}
governed by
\begin{eqnarray}\label{eqn:state2}
	a[y,\upsilon] + b[u,\upsilon]= [f , \upsilon], \qquad \upsilon \in \mathcal{Y},
\end{eqnarray}
where
\begin{subequations}
	\begin{eqnarray*}
		a[y,\upsilon] &=& \mathbb{E} \left[  \int_{\mathcal{D}} \big( a(\boldsymbol{x},\omega) \nabla y \cdot \nabla \upsilon  + \mathbf{b}(\boldsymbol{x},\omega)\cdot \nabla y  \, \upsilon \big)\, d\boldsymbol{x}  \right], \quad \, \forall y, \upsilon \in \mathcal{Y}, \\
		b[u,\upsilon]  &=& -\mathbb{E} \left[  \int_{\mathcal{D}} u  \upsilon \, d\boldsymbol{x} \right] \quad \hbox{and} \quad [f , \upsilon]= \mathbb{E} \left[  \int_{\mathcal{D}} f  \upsilon \, d\boldsymbol{x} \right], \quad \forall u \in \mathcal{U}, \upsilon \in \mathcal{Y}.
	\end{eqnarray*}
\end{subequations}
Moreover, the optimality system in \eqref{eqn:opti} can be stated in the weak formulation as follows:
\begin{subequations}\label{eqn:opti1}
	\begin{align}
		& a[y,\upsilon] + b[u,\upsilon] = [f,\upsilon], \qquad \qquad \qquad \qquad \quad   \upsilon \in \mathcal{Y}, \\
		& a[q,p]  = [y  -y^d,q] +  \gamma \big[y - \mathbb{E}[y] ,q\big], \qquad \quad  \;\;\; q  \in \mathcal{Y},\\
		& \big( \mathbb{E}[p] + \mu u, w-u \big) \geq  0,  \qquad  \quad \quad \qquad \qquad \quad  w \in \mathcal{U}^{ad},
	\end{align}
\end{subequations}	
where the adjoint $ p \in \mathcal{Y} $ solves the following convection diffusion equation containing uncertain inputs
\begin{subequations}\label{eqn:adjoint}
	\begin{align}
		-\nabla \cdot \big( a(\boldsymbol{x},\omega) \nabla p \big) - \mathbf{b}(\boldsymbol{x},\omega)\cdot \nabla p
		& =  (y  -y^d) +  \gamma \big(y - \mathbb{E}[y] \big) \; \hbox{in} \; \mathcal{D} \times \Omega, \\
		p & =  0 \qquad \qquad \hspace{16mm} \; \hbox{on} \; \partial \mathcal{D} \times \Omega.
	\end{align}
\end{subequations}

In the following, we introduce the techniques, that is, Karhunen--L\`{o}eve (KL) expansion, stochastic Galerkin, and discontinuous Galerkin  method, to recast the infinite--dimensional model problem \eqref{eqn:objec2}--\eqref{eqn:state2} into the finite dimensional.


\section{Finite dimensional representation}\label{sec:finite}

\subsection{Finite representation of stochastic fields}

To solve \eqref{eqn:objec2}--\eqref{eqn:state2} numerically, it is needed to reduce the stochastic process into finite mutually uncorrelated random variables. Therefore, the coefficients  $a(\boldsymbol{x},\omega)$ and  $\mathbf{b}(\boldsymbol{x},\omega)$ are approximated by finite uncorrelated components $\{\xi_i(\omega)\}_{i=1}^{N \in \mathbb{N}}$, called as finite dimensional noise \cite{IBabuska_RTempone_GEZouraris_2004a,NWiener_1938a}.  Introducing the probability density functions  of $\{\xi_i(\omega)\}_{i=1}^{N \in \mathbb{N}}$ denoted by $\rho_i \, : \, \Gamma_i \rightarrow [0,1]$ with  a bounded interval $\Gamma_i = \xi_i(\Omega) \in \mathbb{R}$, the probability space $(\Omega, \mathcal{F}, \mathbb{P})$ is replaced by $(\Gamma, \mathcal{B}(\Gamma), \rho(\xi)d\xi)$, where $\Gamma$ represents the support of such probability density,  $\mathcal{B}(\Gamma)$ is a Borel $\sigma$--algebra, and $\rho(\xi)d\xi$ corresponds to the distribution measure of $\xi$. Moreover, $\rho(\xi)$ denotes the joint probability density function. Hence, we can state the tensor--product space $H^k(\mathcal{D}) \otimes L^2(\Gamma)$ endowed with the following norm
\begin{eqnarray}\label{defn:norm}
	\|\eta\|_{H^k(\mathcal{D}) \otimes L^2(\Gamma)} := \left( \int_{\Gamma} \|\eta(\cdot,\xi)\|^2_{H^k(\mathcal{D})} \rho(\xi) \, d \xi\right)^{1/2} < \infty.
\end{eqnarray}

\noindent Following the well--known KL expansion \cite{KKarhunen_1947,MLoeve_1946}, a random field $\eta$ having  a continuous covariance function as follows
\begin{equation}\label{eqn:cov}
	\mathbb{C}_{\eta}(\boldsymbol{x},\boldsymbol{y}) :=  \int_{\Omega}  (\eta(\boldsymbol{x},\cdot) - \overline{\eta}(\boldsymbol{x})) (\eta(\boldsymbol{y},\cdot) - \overline{\eta}(\boldsymbol{y}))  \, d \mathbb{P}(\omega)
\end{equation}
admits a proper orthogonal decomposition
\begin{equation}\label{eqn:kl}
	\eta(\boldsymbol{x},\omega) = \overline{\eta}(\boldsymbol{x}) +  \kappa \sum \limits_{k=1}^\infty \sqrt{\lambda_k}\phi_k(\boldsymbol{x})\xi_k(\omega),
\end{equation}
where $\overline{\eta}(\boldsymbol{x})$ and $\kappa$ are  mean and  standard deviation of $\eta$, respectively, and $\xi:=\{\xi_1,\xi_2,\ldots\}$ are uncorrelated random  variables. The pair $\{\lambda_k,\phi_k\}$ is a set of  the eigenvalues and eigenfunctions of the corresponding covariance operator $\mathcal{C}_{\eta}$. Then, we approximate $\eta(\boldsymbol{x},\omega)$ by truncating its KL expansion of the form
\begin{equation}\label{eqn:kltrun}
	\eta(\boldsymbol{x},\omega) \approx \eta_N(\boldsymbol{x},\omega) := \overline{\eta}(\boldsymbol{x}) + \kappa \sum \limits_{k=1}^N \sqrt{\lambda_k}\phi_k(\boldsymbol{x})\xi_k(\omega).
\end{equation}
The truncated KL expansion \eqref{eqn:kltrun} is a finite representation of the random field $\eta(\boldsymbol{x},\omega)$ in the sense that the mean-square error of approximation is minimized; see, e.g., \cite{IBabuska_PChatzipantelidis_2002a}. To guarantee the positivity of the truncated KL expansion  \eqref{eqn:kltrun} for the diffusivity coefficient $a(\boldsymbol{x},\omega)$, it is also assumed that the mean of random coefficient exhibits a stronger dominance; see, e.g., \cite{CEPowell_HCElman_2009}.

By the assumption on the finite dimensional and Doob--Dynkin lemma \cite{BOksendal_2003}, the solution of \eqref{eqn:m1}  can be expressed  in the finite dimensional stochastic space, that means, $y(\boldsymbol{x},\xi(\omega)) \in \mathcal{Y}_{\rho} =  L^2(H^1_0(\mathcal{D});\Gamma)$ with $\xi = \big(\xi_1(\omega),\ldots,\xi_N(\omega)\big)$.  Then,  setting $\widetilde{\mathbb{E}}[y] = \int_{\Gamma} y \, \rho(\xi)d\xi$, the  optimization problem \eqref{eqn:objec2}-\eqref{eqn:state2} becomes
\begin{eqnarray}\label{eqn:objec4}
	\min \limits_{u \in \mathcal{U}^{ad}} \; \mathcal{J}(u) &=& \frac{1}{2} \widetilde{\mathbb{E}} \left[ \int_{\mathcal{D}} \left( y(u) -y^d\right)^2 \, d\boldsymbol{x} \right] + \frac{\gamma}{2}   \widetilde{\mathbb{E}} \left[ \int_{\mathcal{D}} \left(y(u) - \widetilde{\mathbb{E}}[y(u)]\right)^2 \, d\boldsymbol{x} \right] \nonumber  \\
	&&+ \frac{\mu}{2}  \int_{\mathcal{D}} u^2 \, d\boldsymbol{x}
\end{eqnarray}
subject to
\begin{align}\label{eqn:vari_m2}
	a[y,\upsilon]_{\rho} + b[u,\upsilon]_{\rho}& = [f,\upsilon]_{\rho}, \quad \forall \upsilon \in \mathcal{Y}_{\rho},
\end{align}
where
\begin{subequations}
	\begin{align}
		a[y,\upsilon]_{\rho} &=\hspace{-1.5mm} \int_{\Gamma}  \int_{\mathcal{D}} \hspace{-2mm}\big( a(\boldsymbol{x},\xi) \nabla y \cdot \nabla \upsilon  + \mathbf{b}(\boldsymbol{x},\xi)\cdot \nabla y  \, \upsilon \big)\, d\boldsymbol{x} \, \rho(\xi)d\xi, \; \forall y, \upsilon \in \mathcal{Y}_{\rho}, \\
		b[u,\upsilon]_{\rho}  &= -\int_{\Gamma}  \int_{\mathcal{D}} u  v \, d\boldsymbol{x} \, \rho(\xi)d\xi,  \qquad  \forall u \in \mathcal{U}, \;  \upsilon \in \mathcal{Y}_{\rho}, \\
		[f,\upsilon]_{\rho} &= \int_{\Gamma}  \int_{\mathcal{D}} f  \upsilon \, d\boldsymbol{x} \, \rho(\xi)d\xi,  \qquad  \quad  \forall   \upsilon \in \mathcal{Y}_{\rho}.
	\end{align}
\end{subequations}

Then,  the optimization problem \eqref{eqn:objec4}-\eqref{eqn:vari_m2}  has a unique solution pair $(y,u) \in \mathcal{Y}_{\rho} \times \mathcal{U}^{ad} $ if and only if there is an adjoint  $p \in \mathcal{Y}_{\rho} $ such that the following optimality system holds for the triplet $(y,u,p)$:
\begin{subequations}\label{eqn:opti2}
	\begin{align}
		&	a[y,\upsilon]_{\rho} + b[u,\upsilon]_{\rho}  = [f,\upsilon]_{\rho}, \qquad \qquad \qquad \qquad \;\; \upsilon \in \mathcal{Y}_{\rho} \label{eqn}, \\
		&	a[q,p]_{\rho}  = [y  -y^d,q]_{\rho} +  \gamma \big[ y - \widetilde{\mathbb{E}}[y] ,q\big]_{\rho}, \qquad \quad \; q \in \mathcal{Y}_{\rho},\\
		&	\big( \widetilde{\mathbb{E}} [ p]  + \mu u,  w-u \big) \geq  0,  \qquad \qquad   \qquad  \qquad \qquad    w \in \mathcal{U}^{ad}.
	\end{align}
\end{subequations}

Next, we present the representation of  stochastic solutions, i.e., $y(\boldsymbol{x},\xi), p(\boldsymbol{x},\xi)$,  by using a polynomial chaos (PC) approximation \cite{DXiu_GEKarniadakis_2002}.


\subsection{Stochastic Galerkin Method}

The state solution $y(\boldsymbol{x},\xi) \in L^2(\Gamma, \mathcal{F},\mathbb{P})$, as well as  the adjoint solution $ p(\boldsymbol{x},\xi)  \in L^2(\Gamma, \mathcal{F},\mathbb{P})$, can be represented by a finite generalized polynomial chaos (PC) approximation as stated in Cameron--Martin theorem \cite{RHCameron_WTMartin_1947a},
\begin{subequations}\label{eq:pcekdvtrun}
	\begin{eqnarray}
		y(\boldsymbol{x},\omega) \approx y_J(\boldsymbol{x},\xi) & = & \sum_{i=0}^{J-1} y_i(\boldsymbol{x}) \psi_i(\xi),\\
		p(\boldsymbol{x},\omega) \approx p_J(\boldsymbol{x},\xi) & = & \sum_{i=0}^{J-1} p_i(\boldsymbol{x}) \psi_i(\xi),
	\end{eqnarray}
\end{subequations}
where $y_i(\boldsymbol{x})$ and $p_i(\boldsymbol{x})$ are the deterministic modes of the expansion and the total number of PC basis is determined by the dimension $ N $  of the random vector $\xi$ and the highest order $ Q $  in the basis set of $ \psi_i $
\begin{equation*}	
	J = 1 + \sum \limits_{s=1}^{Q} \frac{1}{s!} \prod \limits_{j=0}^{s-1} (N + j) = \frac{(N+Q)!}{N!Q!}.
\end{equation*}
By following \cite{OGErnst_EUllmann_2010,CEPowell_HCElman_2009},  we then define the stochastic space as
\begin{equation}
	\mathcal{S}_k  := \hbox{span}\{\psi_i(\xi): \; i=0,1,\ldots,J-1\} \subset L^2(\Gamma).
\end{equation}
For simplicity, we only deal with the state equation since the procedure for the adjoint equation is similar to the state ones. By inserting KL expansions \eqref{eqn:kltrun} of the diffusion $a(\boldsymbol{x},\omega)$ and the convection $\mathbf{b}(\boldsymbol{x},\omega)$ coefficients, and the solution expression \eqref{eq:pcekdvtrun} into  the variational form of the state equation \eqref{eqn:vari_m2} and projecting onto the space of  the PC basis functions, we get a linear system, consisting  of $J$ deterministic convection diffusion equations for $j=0, \ldots, J-1$
\begin{eqnarray}\label{sdg1}
	\sum_{i=0}^{J-1} \big( -\nabla \cdot (a_{ij}\nabla y_i(\boldsymbol{x})) +  \mathbf{b}_{ij} \cdot \nabla y_i(\boldsymbol{x}) \big) = \ang{\psi_j}f(\boldsymbol{x}) + \ang{\psi_j}u(\boldsymbol{x}),
\end{eqnarray}
where
\begin{eqnarray*}
	a_{ij} &=& \overline{a}(\boldsymbol{x}) \ang{ \psi_i^2(\xi)} \delta_{ij} + \kappa_a \sum_{k=1}^N  \sqrt{\lambda_k^{a}} \phi_k^{a}(\boldsymbol{x}) \ang{\xi_k \psi_i(\xi) \psi_j(\xi)}, \\
	\mathbf{b}_{ij} &=& \overline{\mathbf{b}}(\boldsymbol{x}) \ang{ \psi_i^2(\xi)} \delta_{ij} + \kappa_{\mathbf{b}}\sum_{k=1}^N  \sqrt{\lambda_k^{\mathbf{b}}} \phi_k^{\mathbf{b}}(\boldsymbol{x}) \ang{\xi_k \psi_i(\xi) \psi_j(\xi)}.
\end{eqnarray*}
Here, we apply  the same distribution for both diffusion and convection random coefficients in order to reduce  the computational effort. However, it can be  possible to use different distributions; see, e.g, \cite{ABarth_AStein_2022} for more discussion. We also note that the quantity of interest is the statistical moments of the solution $y(\boldsymbol{x},\omega)$  rather than  the solution $y(\boldsymbol{x},\omega)$.

\subsection{Symmetric interior penalty Galerkin (SIPG) method}

We briefly recall the SIPG discretization following the studies in  \cite{HYucel_PBenner_2015a,PCiloglu_HYucel_2022}. A shape-regular simplicial triangulations of $\mathcal{D}$ is denoted by
$\{ \mathcal{T}_h\}_h$ with $\overline{\mathcal{D}} = \bigcup_{K \in \mathcal{T}_h} \overline{K}$. The set of all edges $\mathcal{E}_h$ consists of  the interior edges $\mathcal{E}^0_h$  and boundary edges $\mathcal{E}^{\partial}_h$ such that $\mathcal{E}_h=\mathcal{E}^{0}_h \cup\mathcal{E}^{\partial}_h$. For a fixed realization $\omega$  and the unit outward normal $\mathbf{n}_{K}$ to $\partial K$, we decompose the boundary edges  of an element $K$ into the inflow $\partial K^-$
\begin{eqnarray*}
	\partial K^- &=& \{ \boldsymbol{x} \in \partial K \, : \;\mathbf{b}(\boldsymbol{x},\omega) \cdot \mathbf{n}_{K}(\boldsymbol{x}) <0  \}
\end{eqnarray*}
and  outflow  $\partial K^+$  parts such that $\partial K^+ = \partial K \backslash \partial K^-$. Jump and average operators of $y$ and $\nabla y$   for a common edge $E = K  \cap K^e $ are given, respectively, by
\begin{subequations}
	\begin{align}
		\jump{y}&=y\arrowvert_E\mathbf{n}_{K}+y^e\arrowvert_E\mathbf{n}_{K^e}, \quad
		\jump{\nabla y}=\nabla y\arrowvert_E \cdot \mathbf{n}_{K}+\nabla y^e\arrowvert_E \cdot \mathbf{n}_{K^e}, \\
		\average{y}&=\frac{1}{2}\big( y\arrowvert_E+y^e\arrowvert_E \big), \quad
		\average{\nabla y}=\frac{1}{2}\big(\nabla y\arrowvert_E+\nabla y^e\arrowvert_E \big),
	\end{align}
\end{subequations}
where $y\arrowvert_E$ (or $\nabla y\arrowvert_E$)     and  $y^e\arrowvert_E$ (or $\nabla y^e\arrowvert_E$) are traces from inside $K$  and $K^e$, respectively. For a boundary edge $E \in K \cap \partial \mathcal{D}$, the operators are defined by  $\average{\nabla y}=\nabla y$ and $\jump{y}=y\mathbf{n}$, where $\mathbf{n}$ denotes the unit outward normal to  $\partial \mathcal{D}$. Further, we  set $h=\max \limits_{K \in \mathcal{T}_h} h_{K}$, where $h_{K}$ is the diameter of an element $K$.

Defining  the  discrete  space as follows
\begin{align}\label{tspace}
	V_h &= \{y \in L^2(\mathcal{D})\, : \; y \mid_{K}\in \mathbb{P}(K) \quad \forall K \in \mathcal{T}_h\},
\end{align}
where $\mathbb{P}(K)$ is the set of linear polynomials and following the standard  discontinuous Galerkin structure discussed in \cite{DNArnold_FBrezzi_BCockburn_LDMarini_2002a,BRiviere_2008a},  the (bi)--linear forms for a finite dimensional vector $\xi$ can be stated as follow:
\begin{align*}
	a_h(y,v,\xi)&= \sum \limits_{K \in \mathcal{T}_h} \int \limits_{K} a(.,\xi) \nabla y \cdot  \nabla v \, d\boldsymbol{x}
	-  \sum \limits_{ E \in \mathcal{E}^{0}_h \cup \mathcal{E}_h^{\partial}} \int \limits_E \average{a(.,\xi)\nabla y }  \jump{v} \, ds\\
	& \quad - \sum \limits_{ E \in \mathcal{E}^{0}_h \cup \mathcal{E}_h^{\partial}} \int \limits_E \average{a(.,\xi) \nabla v }  \jump{y} \; ds
	+ \sum \limits_{ E \in \mathcal{E}^{0}_h \cup \mathcal{E}_h^{\partial}} \frac{\sigma }{h_E} \int \limits_E \jump{y} \cdot \jump{v} \, ds   \\
	& \quad + \sum \limits_{K \in \mathcal{T}_h} \int \limits_{K} \mathbf {b}(.,\xi) \cdot \nabla y v \; d\boldsymbol{x} \nonumber  +\hspace{-2.5mm} \sum \limits_{K \in \mathcal{T}_h}\; \int \limits_{\partial K^{-} \backslash \partial \mathcal{D}} \hspace{-4mm}\mathbf {b}(.,\xi) \cdot \mathbf{n}_E (y^e-y)v \, ds\\
	& \quad - \sum \limits_{K \in \mathcal{T}_h} \; \int \limits_{\partial K^{-} \cap \partial \mathcal{D}^{-}} \mathbf {b}(.,\xi) \cdot \mathbf{n}_E y v  \, ds, \\
	b_h(u,v,\xi) &= -\sum \limits_{K \in \mathcal{T}_h} \int \limits_{K} u v \, d\boldsymbol{x},	\\
	l_h(f,v,\xi)&=\sum \limits_{K \in \mathcal{T}_h} \int \limits_{K} f v \, d\boldsymbol{x}
	+ \sum \limits_{E \in \mathcal{E}_h^{\partial}} \int \limits_E   \left( \frac{\sigma}{h_E} y_{DB} \, \jump{v}
	-  y_{DB} \, \average{a(.,\xi)\nabla v} \right) \; ds \\
	& \quad - \sum \limits_{K \in \mathcal{T}_h}  \int \limits_{\partial K^{-} \cap \partial \mathcal{D}^{-}} \mathbf{b}(.,\xi) \cdot \mathbf{n}_E  y_{DB} \, v  \, ds,
\end{align*}
where the parameter $\sigma \in \mathbb{R}_0^+$, called as the penalty parameter, should be sufficiently large to ensure the stability of the SIPG scheme; independent of the mesh size $h$. However, as discussed in \cite[Sec.~2.7.1]{BRiviere_2008a}, it depends on the degree of polynomials used in the DG discretization and the position of the edge $E$. In our numerical experiments, we choose $\sigma$ as $\sigma =6$ on the interior edges  $\mathcal{E}^0_h$ and 12 on the boundary edges $\mathcal{E}^{\partial}_h$.

\noindent Then, the (bi)--linear forms of the stochastic discontinuous Galerkin (SDG) for the state equation correspond to
\[
   a_{\xi}[y,v] + b_{\xi}[u,v] = [f,v]_{\xi},
\]
where
\begin{align*}
	a_{\xi}[y,v] &=\int_{\Gamma}  a_h(y,v,\xi)\rho(\xi) \, d\xi, \quad 	b_{\xi}[u,v]=\int_{\Gamma}  b_h(u,v,\xi)\rho(\xi) \, d\xi, \\
	[f,v]_{\xi} &= \int_{\Gamma}  l_h(f,v,\xi)\rho(\xi) \, d\xi.
\end{align*}
Now, we can state the discrete optimal control problem
\begin{eqnarray}\label{eqn:objec_disc}
	\min \limits_{u_h \in \mathcal{U}^{ad}_h} \; \mathcal{J}(u_h) &=& \frac{1}{2} \widetilde{\mathbb{E}} \left[ \int_{\mathcal{D}} \left( y_h -y^d \right)^2 \, d\boldsymbol{x} \right] + \frac{\gamma}{2}   \widetilde{\mathbb{E}} \left[ \int_{\mathcal{D}} \left(y_h - \widetilde{\mathbb{E}}[y_h]\right)^2 \, d\boldsymbol{x} \right] \nonumber  \\
	&&+ \frac{\mu}{2}  \int_{\mathcal{D}} u_h^2 \, d\boldsymbol{x}
\end{eqnarray}
governed by
\begin{equation}\label{eqn:vari_disc}
	a_{\xi}[y_h,\upsilon_h] + b_{\xi}[u_h,\upsilon_h] = [f,\upsilon_h]_{\xi}, \quad \forall \upsilon_h \in \mathcal{Y}_h= V_h \otimes \mathcal{S}_k,
\end{equation}
where the discrete admissible set \eqref{defn:add} is defined by
\begin{equation}\label{defn:add_dis}
	\mathcal{U}_h^{ad} :=  \{ u_h \in \mathcal{U}_h \; : \; u_a \leq  u_h(\boldsymbol{x}) \leq u_b, \; \hbox{a.e.} \;  \boldsymbol{x} \in K  \subset \mathcal{T}_h \},
\end{equation}
with $\mathcal{U}_h^{ad} =\mathcal{U}_h \cap \mathcal{U}^{ad}$ and $\mathcal{U}_h=\mathcal{V}_h$. Analogously, a  pair $(y_h,u_h) \in \mathcal{Y}_h \times \mathcal{U}_h^{ad}$ is a unique solution of the control problem \eqref{eqn:objec_disc}-\eqref{eqn:vari_disc} if and only if  an adjoint $p_h \in \mathcal{Y}_h$ exists  such that the  optimality system holds for $(y_h,u_h,p_h) \in \mathcal{Y}_h \times \mathcal{U}_h^{ad} \times \mathcal{Y}_h $
\begin{subequations}\label{eqn:opti_disc}
	\begin{align}
		&a_{\xi}[y_h,\upsilon_h] + b_{\xi}[u_h,\upsilon_h] = [f,\upsilon_h]_{\xi}, \qquad \qquad \qquad \qquad \quad \; \upsilon_h \in \mathcal{Y}_{h} \label{eqn:Adisc},\\
		&a_{\xi}[q_h,p_h]  = [y_h -y^d,q_h]_{\xi} +  \gamma \big[ y_h - \widetilde{\mathbb{E}}[y_h] ,q_h \big]_{\xi}, \quad \qquad \, q_h \in \mathcal{Y}_{h},\label{eqn:Adjdisc}\\
		& [ p_h + \mu u_h, w_h-u_h ]_{\xi} \geq  0,  \qquad  \qquad\qquad\qquad\qquad\qquad   \;  w_h \in \mathcal{U}_h^{ad}, \label{eqn:opti_disc3}
	\end{align}
\end{subequations}	
where $[ p_h + \mu u_h, w_h-u_h ]_{\xi} = \big( \widetilde{\mathbb{E}}[p_h] + \mu u_h, w_h-u_h \big)$ since the discrete solution $u_h$ is deterministic.

Further, by denoting
\begin{eqnarray}\label{eq:direc_disc}
	\mathcal{J}'_h(u_h) \cdot w_h = [p_h+\mu u_h,w_h]_{\xi}, \quad \forall w_h \in \mathcal{U}_h^{ad},
\end{eqnarray}
one can easily obtain the following expression for the discrete directional derivative of functional $\mathcal{J}_h(u_h)$:
\begin{eqnarray}\label{ineq:direc_disc}
	\mathcal{J}'_h(u_h) \cdot (w_h-u_h) \geq 0, \quad \forall w_h \in \mathcal{U}_h^{ad}.
\end{eqnarray}



\section{Error analysis}\label{sec:apriori}

We provide an a priori error analysis of  the optimization problem  \eqref{eqn:objec2}-\eqref{eqn:state2}, discretized by the  stochastic discontinuous Galerkin method. Before deriving the corresponding estimates, we define the associated energy norm on $\mathcal{D} \times \Gamma$  as
\begin{eqnarray} \label{energynorm}
	\lVert y\rVert _{\xi}= \Bigg(\int_{\Gamma}  \lVert y(.,\xi)\rVert_{e}^2 \; \rho(\xi)\, d\xi \Bigg)^{\frac{1}{2}},
\end{eqnarray}
where $\lVert y(.,\xi)\rVert_{e}$ is the energy norm on $\mathcal{D}$, given as
\begin{eqnarray*}
	\lVert y(.,\xi)\rVert_{e}&=&\Bigg( \sum \limits_{K \in \mathcal{T}_h} \int \limits_{K} a(.,\xi) (\nabla y)^2 \, d\boldsymbol{x}  + \sum \limits_{ E \in \mathcal{E}^{0}_h \cup \mathcal{E}_h^{\partial}} \frac{\sigma}{h_E} \int \limits_E \jump{y}^2 \, ds  \\
	&&	+\frac{1}{2}\sum \limits_{ E \in \mathcal{E}_h^{\partial}}\int \limits_E \mathbf {b}(.,\xi)\cdot \mathbf{n}_E y^2 ds
	+ \frac{1}{2}\sum \limits_{ E \in \mathcal{E}^{0}_h }\int \limits_E \mathbf {b}(.,\xi)\cdot \mathbf{n}_E(y^e-y)^2 \, ds\Bigg)^{\frac{1}{2}}.
\end{eqnarray*}
By the standard arguments as done in deterministic case, one can easily show the coercivity and continuity of $a_{\xi}(\cdot,\cdot)$ for $y, v \in \mathcal{Y}_h$
\begin{equation}\label{coer_cont}
	a_{\xi}[y,y]  \geq  c_{cv} \, \lVert y \rVert _{\xi}^2, \qquad
	a_{\xi}[y,v]  \leq  c_{ct} \, \lVert y\rVert _{\xi}\lVert v\rVert _{\xi},
\end{equation}
where the coercivity constant $c_{cv}$ depends on $a_{min}$, whereas the continuity constant $c_{ct}$ depends on $a_{max}$.

Next, we state the estimates on the finite dimensional probability domain $\Gamma$ and the physical domain $K\in \mathcal{T}_h$. Let a partition of the support of probability density in finite dimensional space, i.e.,  $\Gamma= \prod \limits_{n=1}^{N} \Gamma_n$, consists of  disjoint $\mathbf{R}^N$--boxes, $\gamma = \prod \limits_{n=1}^{N} (r_n^{\gamma},s_n^{\gamma})$, with $(r_n^{\gamma},s_n^{\gamma}) \subset \Gamma_n$ for $n=1,\ldots,N$ so that the mesh size $k_n$ becomes $k_n = \max \limits_{\gamma} \lvert s_n^{\gamma} - r_n^{\gamma}\rvert$, $n=1 \ldots N$. For the multi--index $q = (q_1,\ldots,q_N)$, the (discontinuous) finite element approximation space having  at most $q_n$ degree on each direction $\xi_n$ is denoted by $\mathcal{S}_k^q \subset L^2(\Gamma)$. Then, for $v\in H^{q+1}(\Gamma), \varphi \in \mathcal{S}_k^q$, we have, see \cite{IBabuska_RTempone_GEZouraris_2004a},
\begin{eqnarray}\label{estr}
	\min \limits_{\varphi \in \mathcal{S}_k^q} \|v-\varphi\|_{L^2(\Gamma)} \leq   \sum_{n=1}^{N}\bigg(\frac{k_n}{2}\bigg)^{q_n+1}\dfrac{\lVert \partial^{q_n +1}_{\xi_n}v \rVert_{L^2(\Gamma)} }{(q_n+1)!}.
\end{eqnarray}
For $v\in H^2(K)$ and $\widetilde{v}  \in \mathbb{P}(K)$, where $K\in \mathcal{T}_h$, the following discontinuous Galerkin approximation \cite[Theorem 2.6]{BRiviere_2008a} also holds
\begin{equation} \label{approxfm}
	\lVert v- \widetilde{v} \rVert_{H^q(K)} \leq  C\,h^{2-q} \lvert v \rvert_{H^2(K)}, \qquad 0\leq q \leq 2,
\end{equation}
where the constant $C$ is not depending  on $v$ and $h$.

\indent Further, we define the following projection operators, which are needed in the rest of the paper:
\begin{itemize}
	\item $L^2$--projection operators  $\Pi_n: L^2(\Gamma) \rightarrow \mathcal{S}_k^q$ and  $\Pi_h: L^2(\mathcal{D}) \rightarrow V_h \cap L^2(\mathcal{D})$ are  given by
	\begin{subequations}
		\begin{eqnarray}
			(\Pi_n(\xi)-\xi,\zeta)_{L^2(\Gamma)}       &=&0, \qquad  \forall \zeta \in \mathcal{S}_k^q, \qquad \forall\xi \in L^2(\Gamma), \label{eq:l2projc}\\
			(\Pi_h( \nu)- \nu,\chi)_{L^2(\mathcal{D})} &=&0, \qquad  \forall \chi\in V_h, \qquad \forall \nu \in L^2(\mathcal{D}). \label{eq:l2projc_phy}
		\end{eqnarray}
	\end{subequations}
	with the following estimate
	\begin{equation}\label{lprojest}
		\|\nu - \Pi_h( \nu)\|_{L^2(L^2(\mathcal{D};\Gamma))} \leq C  h \|\nu\|_{L^2(H^1(\mathcal{D};\Gamma))}.
	\end{equation}
    In addition, taking $\zeta=\Pi_n(\xi)$ and $\chi = \Pi_h( \nu)$ in \eqref{eq:l2projc} and \eqref{eq:l2projc_phy}, respectively, it holds that
    \begin{subequations}
    	\begin{eqnarray}
    		\|\Pi_n(\xi)\|_{L^2(\Gamma)}       &\leq& C	\|\xi\|_{L^2(\Gamma)}, \qquad  \forall \xi \in L^2(\Gamma),  \label{eq:l2bound}\\
    		\| \Pi_h( \nu)\|_{L^2(\mathcal{D})} &\leq& C\|\nu\|_{L^2(\mathcal{D})}, \qquad  \forall \nu \in L^2(\mathcal{D}). \label{eq:l2bound_phy}
    	\end{eqnarray}
    \end{subequations} 
	\item $H^1$--projection operator $\mathcal{R}_h:H^1(\mathcal{D}) \rightarrow V_h \cap H^1(\mathcal{D})$  is stated by
	\begin{subequations} \label{Hproj_main}
		\begin{align}
			(\mathcal{R}_h( \upsilon)- \upsilon,\vartheta)_{L^2(\mathcal{D})}&=0, \qquad \quad \forall \vartheta \in V_h, \qquad \forall \upsilon \in H^1(\mathcal{D}), \label{eq:Hprojc} \\
			(\nabla(\mathcal{R}_h( \upsilon)- \upsilon),\nabla \vartheta)_{L^2(\mathcal{D})}&=0,  \qquad \quad  \forall \vartheta \in V_h, \qquad \forall \upsilon \in H^1(\mathcal{D}). \label{eq:Hproj}
		\end{align}
	\end{subequations}
\end{itemize}
With the help of the  $H^1$--projection operator in \eqref{eq:Hprojc}, the  Cauchy--Schwarz inequality, the $L^2$--projection operator in \eqref{eq:l2projc}, and the approximation in \eqref{estr},  we obtain the  approximation property (\cite[Theorem 3.2]{PCiloglu_HYucel_2022}): for all $v \in L^2(H^2(\mathcal{D});\Gamma)\cap H^{q+1}(H^1(\mathcal{D});\Gamma)$ and $\widetilde{v} \in V_h \times \mathcal{S}_k^q$
\begin{eqnarray} \label{ineq:main_appr}
	\lVert v- \widetilde{v} \rVert_{L^2(H^1(\mathcal{D});\Gamma)}
	&\leq& Ch \lVert v \rVert_{L^2(H^2(\mathcal{D});\Gamma)} \nonumber \\
	&& + \sum_{n=1}^{N}\bigg(\frac{k_n}{2}\bigg)^{q_n+1}\dfrac{\lVert \partial^{q_n+1}_{\xi_n}v \rVert_{L^2(H^1(\mathcal{D});\Gamma)} }{(q_n+1)!},
\end{eqnarray}
where the constant $C$ does not depend on $v$, $h$, and $k_n$.

To recognize error contributions emerging from the spatial domain $\mathcal{D}$ and the probability domain $\Gamma$, separately, a projection operator $ \mathcal{P}_{hn} $  mapping onto the tensor product space $\mathcal{Y}_{h}$ is given by
\begin{equation}\label{2proj}
	\mathcal{P}_{hn} \Upsilon = \Pi_h \Pi_n \Upsilon = \Pi_n \Pi_h \Upsilon, \quad \forall \Upsilon \in L^2(L^2(\mathcal{D});\Gamma)
\end{equation}
and the  decomposition
\begin{equation}
	\Upsilon - \mathcal{P}_{hn} \Upsilon =(\Upsilon - \Pi_h\Upsilon) + \Pi_h(I-\Pi_n)\Upsilon, \quad \forall \Upsilon \in L^2(L^2(\mathcal{D});\Gamma).
\end{equation}
Then, it follows from \eqref{eq:l2bound}, \eqref{eq:l2bound_phy}, and \eqref{2proj} that
\begin{eqnarray}\label{ineq:boundproc}
 \|\mathcal{P}_{hn} \Upsilon\|_{L^2(L^2(\mathcal{D});\Gamma)} \leq C \| \Upsilon\|_{L^2(L^2(\mathcal{D});\Gamma)}, \quad \forall \Upsilon \in L^2(L^2(\mathcal{D});\Gamma).
\end{eqnarray}

\noindent Before the derivation of a priori error estimate, we state the following auxiliary problem
\begin{eqnarray}\label{eq:direc_semi}
	\mathcal{J}'_h(u) \cdot (w-u) = [p_h(u)+\mu u,w-u]_{\xi}\geq 0, \quad \forall w \in \mathcal{U}^{ad},
\end{eqnarray}
where $p_h(u) \in \mathcal{Y}_{h} $ solves  the following auxiliary system:
\begin{subequations}\label{eqn:opti_semi}
	\begin{align}
		& a_{\xi}[y_h(u),v_h] + b_{\xi}[u,v_h] = [f,v_h]_{\xi}, \qquad \qquad \qquad \qquad \qquad\;\;\; v_h \in \mathcal{Y}_{h} \label{eqn:Asemi}, \\
		& a_{\xi}[q_h,p_h(u)]  = [y_h(u) -y^d,q_h]_{\xi} +  \gamma \big[ y_h(u) - \widetilde{\mathbb{E}}[y_h(u)], q_h \big]_{\xi}, \,   q_h \in \mathcal{Y}_{h}.\label{eqn:Adjsemi}
	\end{align}
\end{subequations}
It is also noted  that we prefer to use $\|u\|_{L^2(L^2(\mathcal{D});\Gamma)}$  in the derivation of error estimates instead of $\|u\|_{L^2(\mathcal{D})}$ for better readability in terms of notation.
\begin{lemma}\label{lemma:direc}
	With the definition in \eqref{eq:direc_semi}, the following estimate holds:
	\begin{equation}\label{eqn:lemma_direc}
		\left( \mathcal{J}'_h(w) - \mathcal{J}'_h(u)  \right) \cdot (w-u)  \geq \mu \lVert w-u \rVert^2 _{L^2(L^2(\mathcal{D});\Gamma)}.
	\end{equation}
\end{lemma}
\begin{proof}
	By \eqref{eq:direc_semi}, we have
	\begin{eqnarray}\label{eqn:lemma1}
		\big( \mathcal{J}'_h(w) -  \mathcal{J}'_h(u) \big) \cdot (w-u)  = [p_h(w)-p_h(u),w-u]_{\xi} + \mu [w-u,w-u]_{\xi}.
	\end{eqnarray}
	Now, it follows from \eqref{eqn:opti_semi} that
	\begin{align}\label{ineq:pwpu}
		[p_h(w)-p_h(u),w-u]_{\xi}  &= a_{\xi}[y_h(w)-y_h(u),p_h(w)-p_h(u)]  \nonumber \\
		&  =  (1+\gamma)[y_h(w)-y_h(u),y_h(w)-y_h(u)]_{\xi}  \nonumber  \\
		& \quad - \gamma \big[ \widetilde{\mathbb{E}}[y_h(w)-y_h(u)] ,y_h(w)-y_h(u)\big]_{\xi}.
	\end{align}
	The usage of  Cauchy-Schwarz and Young's inequalities yields
	\begin{align*}
		& - \gamma \big[ \widetilde{\mathbb{E}}[y_h(w)-y_h(u)] ,y_h(w)-y_h(u)\big]_{\xi} \\
		& \qquad \geq - \frac{\gamma}{2}\lVert \widetilde{\mathbb{E}}[y_h(w)-y_h(u)] \rVert^2_{L^2(L^2(\mathcal{D});\Gamma)} - \frac{\gamma}{2} \lVert y_h(w)-y_h(u) \rVert^2_{L^2(L^2(\mathcal{D});\Gamma)}.\nonumber
	\end{align*}
	Since all norms are convex functions, Jensen's inequality $ \lVert \widetilde{\mathbb{E}}[u] \rVert \leq \widetilde{\mathbb{E}}\lVert u \rVert$ and $\widetilde{\mathbb{E}}[\widetilde{\mathbb{E}}[u]] = \widetilde{\mathbb{E}}[u]$ give us
	\begin{equation}\label{ineq:negGamma}
		- \gamma \big[ \widetilde{\mathbb{E}}[y_h(w)-y_h(u)], y_h(w)-y_h(u)\big]_{\xi}
		  \geq - \gamma \lVert y_h(w)-y_h(u) \rVert^2_{L^2(L^2(\mathcal{D});\Gamma)}.
	\end{equation}
	Thus, inserting \eqref{ineq:negGamma} into \eqref{ineq:pwpu}, it is obtained that
	\begin{align}\label{eqn:lemma2}
		[p_h(w)-p_h(u),w-u]_{\xi} \geq \underbrace{\lVert y_h(w)-y_h(u) \rVert^2_{L^2(L^2(\mathcal{D});\Gamma)}}_{\geq 0}.
	\end{align}
	Hence, \eqref{eqn:lemma1} and \eqref{eqn:lemma2} imply that \eqref{eqn:lemma_direc} holds.
\end{proof}

Next, we derive  an upper  bound for the error  between the discrete solutions $(y_h,p_h)$ and the auxiliary solutions $(y_h(u),p_h(u))$.
\begin{lemma}\label{lemma:y_hp_h}
	Assume that $(y_h,p_h)$ and $(y_h(u),p_h(u))$, respectively, are the solutions of \eqref{eqn:opti_disc} and \eqref{eqn:opti_semi}. Then,  the following estimates exist for
	positive constants $C_1$ and $C_2$ independent of $h$
	\begin{subequations}
		\begin{eqnarray}
			\lVert y_h -y_h(u) \rVert _{\xi} &\leq& C_1 \lVert u-u_h \rVert _{L^2(L^2(\mathcal{D});\Gamma)}, \label{ineq:y_h} \\
			\lVert p_h- p_h(u) \rVert_{\xi} &\leq& C_2 \lVert u -u_h \rVert_{L^2(L^2(\mathcal{D});\Gamma)}.  \label{ineq:p_h}
		\end{eqnarray}
	\end{subequations}
\end{lemma}
\begin{proof}
	By subtracting \eqref{eqn:Asemi} from \eqref{eqn:Adisc} and taking $v_h = y_h - y_h(u)$, we have that
	\begin{eqnarray*}
		a_{\xi}[y_h -y_h(u),y_h -y_h(u)] = [u_h-u,y_h -y_h(u)]_{\xi}.
	\end{eqnarray*}
	With the help of  the coercivity of $a_{\xi}$ \eqref{coer_cont} and the Cauchy-Schwarz inequality, we obtain
	\begin{align*}
		c_{cv}\lVert y_h -y_h(u) \rVert _{\xi}^2	&\leq a_{\xi}[y_h -y_h(u),y_h -y_h(u)] \\
		&\leq \lVert u_h-u \rVert _{L^2(L^2(\mathcal{D});\Gamma)} \lVert  y_h -y_h(u) \rVert _{\xi},
	\end{align*}
	which yields the desired result \eqref{ineq:y_h}.
	
	Analogously, by subtracting \eqref{eqn:Adjsemi} from \eqref{eqn:Adjdisc} and taking $v_h = p_h - p_h(u)$, we have that
	\begin{align*}
		&a_{\xi}[p_h - p_h(u),p_h - p_h(u)] \\
		& \qquad =(1+\gamma)[y_h -y_h(u),p_h - p_h(u)]_{\xi}+\gamma \left[\widetilde{\mathbb{E}}[y_h(u) -y_h],p_h - p_h(u)\right]_{\xi}.
	\end{align*}
	It follows from the coercivity of $a_{\xi}$, Cauchy-Schwarz inequality, and Jensen's inequality that
	\begin{align}\label{ineq:ppp}
		c_{cv}\lVert p_h -p_h(u) \rVert _{\xi}^2 & \leq a_{\xi}[p_h - p_h(u),p_h - p_h(u)]  \nonumber \\
		& \leq (1+2\gamma)\lVert p_h - p_h(u) \rVert _{L^2(L^2(\mathcal{D});\Gamma)} \lVert  y_h -y_h(u) \rVert _{\xi}.
	\end{align}
	We note that the procedure applied in \eqref{ineq:negGamma} is also used in the derivation of \eqref{ineq:ppp}. Hence, by \eqref{ineq:ppp} and \eqref{ineq:y_h}, we deduce the desired result \eqref{ineq:p_h}.
\end{proof}

To obtain an upper bound for the control, we divide the domain $\mathcal{D}$ into pieces by considering the active and inactive parts of the control $u$  as done in \cite{RLi_WLiu_HMa_TTang_2002,TAkman_HYucel_BKarasozen_2014}:
\vspace{-4mm}
\begin{subequations}\label{dom}
	\begin{eqnarray}
		\mathcal{D}^{+} &=& \Biggl\{ \bigcup_{K}: K \subset \mathcal{D}, \; u_a < u\arrowvert_{K} < u_b\Biggr\},\\
		\mathcal{D}^{\partial} &=& \Biggl\{ \bigcup_{K}: K \subset \mathcal{D}, \; u\arrowvert_{K} = u_a \; \hbox{or} \; u\arrowvert_{K} = u_b\Biggr\},\\
		\mathcal{D}^{-} &=& \mathcal{D}\setminus (\mathcal{D}^{+} \cup \mathcal{D}^{\partial} ).
	\end{eqnarray}
\end{subequations}
It is assumed that these sets are disjoint, $\mathcal{D} = \mathcal{D}^{+}\cup \mathcal{D}^{\partial} \cup \mathcal{D}^{-} $, and  $\mathcal{D}^{-}$ satisfies the following inequality related to the regularity of $ u $ and $ \mathcal{T}_h$
\begin{eqnarray}
	\hbox{meas}(\mathcal{D}^{-}) \leq Ch,
\end{eqnarray}
which is valid if the boundary of the  $ \mathcal{D}^{\partial} $ is represented by finite rectifiable curves \cite{DMeidner_BVexler_2008a}. Further, we  define a set such that
$
\mathcal{D}^{+} \subset \mathcal{D}^{*} = \{ \boldsymbol{x} \in \mathcal{D} \; : \; u_a < u(\boldsymbol{x}) < u_b\}
$
\cite{ZZhou_NYan_2010a}.

\begin{lemma}\label{lemma:control}
	Let $(y,u,p)$ and $(y_h,u_h,p_h)$, respectively, be  the solutions of \eqref{eqn:opti1} and \eqref{eqn:opti_disc}. Assume that $ u \in L^2(W^{1,\infty}(\mathcal{D});\Gamma)$ with  $ u\arrowvert_{\mathcal{D}^{+}} \in L^2(H^2(\mathcal{D}^{+});\Gamma)$. Then, it holds that
	\begin{align}\label{ineq:control}
		& \lVert u -u_h \rVert _{L^2(L^2(\mathcal{D});\Gamma)} \nonumber \\
		& \quad \leq  C \lVert p-p_h(u) \rVert _{L^2(L^2(\mathcal{D};\Gamma))} + C h^{3/2} \|u\|_{L^2(W^{1,\infty}(\mathcal{D});\Gamma)}  \nonumber \\
		& \qquad + C \Big( h \|p\|_{L^2(H^1(\mathcal{D});\Gamma)} + \sum_{n=1}^{N}\bigg(\frac{k_n}{2}\bigg)^{q_n+1}\dfrac{ \lVert \partial^{q_n+1}_{\xi_n} p \rVert_{L^2(H^1({\mathcal{D}};\Gamma)}}{(q_n+1)!} \Big).
	\end{align}
\end{lemma}
\begin{proof}
	With the help of Lemma \ref{lemma:direc}, \eqref{eq:direc_semi}, the standard Lagrangian interpolation $\Pi u$, the assumption $\mathcal{D}^{+} \subset \mathcal{D}^{*}$, and the notation $p_h=p_h(u_h)$,  we obtain
	\begin{align}\label{ineq:lem1}
		\mu \lVert u-u_h \rVert^2 _{L^2(L^2(\mathcal{D});\Gamma)} &\leq  \mathcal{J}'_h(u) \cdot (u-u_h ) - \mathcal{J}'_h(u_h) \cdot (u-u_h) \nonumber \\
		& = [\mu u + p_h(u),u-u_h]_{\xi} - [\mu u_h + p_h,u-u_h]_{\xi} \nonumber \\
		& = \underbrace{[\mu u + p,u-u_h]_{\xi}}_{-\mathcal{J}'(u) \cdot (u_h-u) \leq 0} - [p-p_h(u),u-u_h]_{\xi} \nonumber \\
		& \quad + \underbrace{[\mu u_h + p_h,u_h - \Pi u]_{\xi}}_{-\mathcal{J}_h'(u_h) \cdot (\Pi u-u_h) \leq 0} + [\mu u_h + p_h,\Pi u -u]_{\xi} \nonumber \\
		& \leq [\mu u_h + p_h,\Pi u -u]_{\xi} + [p_h(u)-p,u-u_h]_{\xi}.
	\end{align}
	The first term in \eqref{ineq:lem1} can be rewritten as follows
	\begin{align}\label{ineq:lem2}
		[\mu u_h + p_h,\Pi u -u]_{\xi} & = [\mu u_h + p_h -\mu u- p,\Pi u -u]_{\xi} + [\mu u + p, \Pi u -u]_{\xi} \nonumber\\
		& = [\mu u_h -\mu u,\Pi u -u]_{\xi} +  [\mu u + p, \Pi u -u]_{\xi}\nonumber \\
		& \quad + [ p_h - p_h(u),\Pi u -u]_{\xi} + [p_h(u)-p, \Pi u -u]_{\xi}.
	\end{align}
	Then, inserting \eqref{ineq:lem2} into \eqref{ineq:lem1} and applying Cauchy-Schwarz and Young's inequalities and Lemma \ref{lemma:y_hp_h}, we obtain
	\begin{eqnarray}\label{ineq:lem3}
		 \mu \lVert u-u_h \rVert^2 _{L^2(L^2(\mathcal{D});\Gamma)}
		&\leq& c_1\lVert p_h(u)-p \rVert_{L^2(L^2(\mathcal{D});\Gamma)}^2 + c_2\lVert u-u_h \rVert_{L^2(L^2(\mathcal{D});\Gamma)}^2  \nonumber \\
		&& + c_3\lVert u - \Pi u \rVert_{L^2(L^2(\mathcal{D});\Gamma)}^2 + [\mu u + p, \Pi u -u]_{\xi}.
	\end{eqnarray}
	Since $ \Pi u(\boldsymbol{x}) = u(\boldsymbol{x}) $ for any vertex $\boldsymbol{x} $,  $ \Pi u \in \mathcal{U}_h^{ad}$  and the following estimates hold
	\begin{subequations}\label{ineq:lem4}
		\begin{eqnarray}
			\lVert u - \Pi u \rVert_{L^2(L^2(\mathcal{D}^{+});\Gamma)} &\leq& Ch^2\lVert u \rVert_{L^2(H^2(\mathcal{D}^{+});\Gamma)}, \label{ineq:lem4a} \\
			\lVert u - \Pi u \rVert_{L^2(W^{0,\infty}(\mathcal{D}^{-});\Gamma)} &\leq& Ch\lVert u \rVert_{L^2(W^{1,\infty}(\mathcal{D}^{-});\Gamma)} \label{ineq:lem4b}
		\end{eqnarray}
	\end{subequations}
	for $ u \in L^2(W^{1,\infty}(\mathcal{D});\Gamma) $ and $u \arrowvert_{\mathcal{D}^*} \subset L^2( H^2(\mathcal{D}^*);\Gamma)$. Hence
	\begin{align}\label{ineq:lem5}
		& \lVert u - \Pi u \rVert_{L^2(L^2(\mathcal{D});\Gamma)}^2 \nonumber \\
		& \quad = \lVert u - \Pi u \rVert_{L^2(L^2(\mathcal{D}^{+});\Gamma)}^2 + \underbrace{\lVert u - \Pi u \rVert_{L^2(L^2(\mathcal{D}^{\partial});\Gamma)}^2}_{=0} + \lVert u - \Pi u \rVert_{L^2(L^2(\mathcal{D}^{-});\Gamma)}^2 \nonumber \\
		& \quad \leq  \lVert u - \Pi u \rVert_{L^2(L^2(\mathcal{D}^{+});\Gamma)}^2  + C\lVert u - \Pi u \rVert_{L^2(W^{0,\infty}(\mathcal{D}^{-});\Gamma)}^2\hbox{meas}(\mathcal{D}^{-}) \nonumber \\
		& \quad \leq  Ch^4 \lVert u  \rVert_{L^2(H^2(\mathcal{D}^{+});\Gamma)}^2 + Ch^3\lVert u \rVert_{L^2(W^{1,\infty}(\mathcal{D}^{-});\Gamma)}^2 \nonumber \\
		& \quad \leq Ch^3 \left( h \lVert u  \rVert_{L^2(H^2(\mathcal{D}^{+});\Gamma)}^2 + \lVert u \rVert_{L^2(W^{1,\infty}(\mathcal{D}^{-});\Gamma)}^2 \right)\nonumber \\
		& \quad \leq Ch^3 \left( \lVert u  \rVert_{L^2(H^2(\mathcal{D}^{+});\Gamma)}^2 + \lVert u \rVert_{L^2(W^{1,\infty}(\mathcal{D}^{-});\Gamma)}^2 \right).
	\end{align}
	By the variational inequality  \eqref{eqn:opti_disc3} and the definitions of domains \eqref{dom}, we have
	\begin{align*}
		\mu u + p = 0 \; \hbox{on} \;\mathcal{D}^{+} \quad \hbox{and} \quad \Pi  u - u = 0 \;  \hbox{on} \;\mathcal{D}^{\partial}.
	\end{align*}
	Then,
	\begin{align}\label{ineq:lem6}
		[\mu u + p, \Pi u -u]_{\xi}  = & \underbrace{ [\mu u - \Pi_h(\mu u ) + \Pi_h(\mu u ), \Pi u -u]_{\mathcal{D}^{-}}}_{T_1}        \nonumber \\
	 & +   \underbrace{[p - \mathcal{P}_{hn}(p) +\mathcal{P}_{hn}(p), \Pi u -u]_{\mathcal{D}^{-}}}_{T_2}.
	\end{align}
	It follows from the inequalities \eqref{lprojest}, \eqref{eq:l2bound_phy}, and \eqref{ineq:lem4b}, Sobolev embedding theorem, see, e.g., \cite{RAAdams_1975}, and Young's inequality that
	\begin{align}\label{ineq:lem7}
		T_1 &= [\mu u - \Pi_h(\mu u ), \Pi u -u]_{\mathcal{D}^{-}} + [ \Pi_h(\mu u ), \Pi u -u]_{\mathcal{D}^{-}}\nonumber \\
		&   \leq  \mu \Big(\|u- \Pi_hu\|_{L^2(L^2(\mathcal{D}^{-});\Gamma)}+ \|\Pi_hu\|_{L^2(L^2(\mathcal{D}^{-});\Gamma)} \Big) \|u - \Pi u\|_{L^2(L^2(\mathcal{D}^{-});\Gamma)} \nonumber \\
		&   \leq  \mu \Big(\|u- \Pi_hu\|_{L^2(L^2(\mathcal{D}^{-});\Gamma)}+ C\|u\|_{L^2(L^2(\mathcal{D}^{-});\Gamma)} \Big) \|u - \Pi u\|_{L^2(L^2(\mathcal{D}^{-});\Gamma)} \nonumber \\
        &   \leq  \mu \Big(\|u- \Pi_hu\|_{L^2(L^2(\mathcal{D}^{-});\Gamma)}+ C\|u\|_{L^2(W^{0,\infty}(\mathcal{D}^{-});\Gamma)}\hbox{meas}(\mathcal{D}^{-}) \Big) \nonumber \\
        &   \qquad \times \|u - \Pi u\|_{L^2(L^2(\mathcal{D}^{-});\Gamma)} \nonumber \\
		&   \leq  C h \|u\|_{L^2(H^1(\mathcal{D}^{-});\Gamma)}  \lVert u - \Pi u \rVert_{L^2(W^{0,\infty}(\mathcal{D}^{-});\Gamma)}\hbox{meas}(\mathcal{D}^{-}) \nonumber \\
		&   \leq  C h^3 \|u\|_{L^2(H^1(\mathcal{D}^{-});\Gamma)} \lVert u \rVert_{L^2(W^{1,\infty}(\mathcal{D}^{-});\Gamma)} \nonumber \\
        &   \leq  C h^3 \Big( \|u\|^2_{L^2(H^1(\mathcal{D}^{-});\Gamma)} + \lVert u \rVert_{L^2(W^{1,\infty}(\mathcal{D}^{-});\Gamma)}^2 \Big).
	\end{align}
	Next, with the help of  the projector operator in \eqref{2proj} and the bounds in \eqref{estr},\eqref{lprojest}, \eqref{ineq:boundproc}, and \eqref{ineq:lem4b}, Sobolev embedding theorem, and Cauchy and Young's inequalities,  we find  a bound for the second term $T_2$ in \eqref{ineq:lem6}
	\begin{align}\label{ineq:lem8}
		T_2 &= [p - \Pi_{h}(p), \Pi u -u]_{\mathcal{D}^{-}}  + [\mathcal{P}_{hn}(p), \Pi u -u]_{\mathcal{D}^{-}} + [\Pi_{h}(I-\Pi_n)(p), \Pi u -u]_{\mathcal{D}^{-}}   \nonumber \\
		&  \leq  \Big(\|p - \Pi_{h}(p)\|_{L^2(L^2(\mathcal{D}^{-});\Gamma)} +  \|\mathcal{P}_{hn}(p)\|_{L^2(L^2(\mathcal{D}^{-});\Gamma)} \Big) \|\Pi u -u\|_{L^2(L^2(\mathcal{D}^{-});\Gamma)}  \nonumber \\
		& \quad  +  \|\Pi_{h}(I-\Pi_n)(p)\|_{L^2(L^2(\mathcal{D}^{-});\Gamma)}  \|\Pi u -u\|_{L^2(L^2(\mathcal{D}^{-});\Gamma)} \nonumber \\
        & \leq   C_1 \Big(  h \|p\|_{L^2(H^1(\mathcal{D}^{-});\Gamma)} + \|p\|_{L^2(W^{0,\infty}(\mathcal{D}^{-});\Gamma)}\hbox{meas}(\mathcal{D}^{-})  \Big) h^2 \lVert u \rVert_{L^2(W^{1,\infty}(\mathcal{D}^{-});\Gamma)} \nonumber \\
        & \quad  +  C_2 \sum_{n=1}^{N}\bigg(\frac{k_n}{2}\bigg)^{q_n+1}\dfrac{ \lVert \partial^{q_n+1}_{\xi_n} p \rVert_{L^2(H^1(\mathcal{D^-});\Gamma)} }{(q_n+1)!} h^2 \lVert u \rVert_{L^2(W^{1,\infty}(\mathcal{D}^{-});\Gamma)} \nonumber \\
        & \leq  C_1 \Big( \frac{h^2}{2} \|p\|_{L^2(H^1(\mathcal{D}^{-});\Gamma)}^2 + \frac{h^4}{2}\lVert u \rVert_{L^2(W^{1,\infty}(\mathcal{D}^{-});\Gamma)}^2 \Big) \nonumber \\
		& \quad + C_2 \Bigg( \frac{1}{2}\sum_{n=1}^{N}\bigg(\frac{k_n}{2}\bigg)^{2q_n+2}\dfrac{ \lVert \partial^{q_n+1}_{\xi_n} p \rVert_{L^2(H^1(\mathcal{D^-});\Gamma)}^2}{((q_n+1)!)^2}
		+ \frac{h^4}{2} \lVert u \rVert_{L^2(W^{1,\infty}(\mathcal{D}^{-});\Gamma)}^2 \Bigg).
	\end{align}
	Combination of  \eqref{ineq:lem7} and  \eqref{ineq:lem8} yields
	\begin{align}\label{ineq:lem9}
		& [\mu u + p, \Pi u -u]_{\xi}  \nonumber \\
        & \quad \leq  C h^3 \Big( \|u\|^2_{L^2(H^1(\mathcal{D}^{-});\Gamma)} + \lVert u \rVert_{L^2(W^{1,\infty}(\mathcal{D}^{-});\Gamma)}^2 \Big) \nonumber \\
		& \qquad   + C h^2 \|p\|_{L^2(H^1(\mathcal{D}^{-});\Gamma)}^2 + C \sum_{n=1}^{N}\bigg(\frac{k_n}{2}\bigg)^{2q_n+2}\dfrac{ \lVert \partial^{q_n+1}_{\xi_n} p \rVert_{L^2(H^1(\mathcal{D^-});\Gamma)}^2}{((q_n+1)!)^2}.
	\end{align}
	Finally, inserting \eqref{ineq:lem5} and \eqref{ineq:lem9} into \eqref{ineq:lem3}, we completes the proof of Lemma~\ref{lemma:control}.
\end{proof}

\begin{lemma}\label{lemma:state_adjo}
	Assume that $ (y,p) $ and $ (y_h(u),p_h(u)) $, respectively,  are the solutions of \eqref{eqn:opti1} and \eqref{eqn:opti_semi}. Then, we have
	\begin{align}\label{ineq:y_hy}
		\lVert y -y_h(u) \rVert _{\xi} &\leq Ch \lVert y \rVert_{L^2(H^2(\mathcal{D});\Gamma)} \nonumber \\
		&\quad + \sum_{n=1}^{N}\bigg(\frac{k_n}{2}\bigg)^{q_n+1}\dfrac{\lVert \partial^{q_n+1}_{\xi_n} y \rVert_{L^2(H^1(\mathcal{D});\Gamma)} }{(q_n+1)!}
	\end{align}
	and
	\begin{align}\label{ineq:p_hp}
		& \lVert p- p_h(u) \rVert_{\xi}\nonumber \\
		& \;\leq Ch  \Big( \lVert y \rVert_{L^2(H^2(\mathcal{D});\Gamma)} +\lVert p \rVert_{L^2(H^2(\mathcal{D});\Gamma)} \Big) \nonumber \\
		&\quad + \sum_{n=1}^{N}\bigg(\frac{k_n}{2}\bigg)^{q_n+1}\dfrac{\Big( \lVert \partial^{q_n+1}_{\xi_n} y \rVert_{L^2(H^1(\mathcal{D});\Gamma)} + \lVert \partial^{q_n+1}_{\xi_n} p \rVert_{L^2(H^1(\mathcal{D});\Gamma)}\Big) }{(q_n+1)!}.
	\end{align}
\end{lemma}
\begin{proof}
	An application of  the coercivity and continuity of  $a_{\xi}$ in \eqref{coer_cont}, $H^1(\mathcal{D})$--projection $\mathcal{R}_h$ in \eqref{Hproj_main}, $L^2(\mathcal{D})$--projection $\Pi_n$ in \eqref{eq:l2projc}, and Galerkin orthogonality yields
	\begin{align*}
		&   c_{cv}\lVert y -y_h(u) \rVert _{\xi}^2	\nonumber \\
		& \quad \leq a_{\xi}[y -y_h(u),y- y_h(u)] \\
		& \quad \leq a_{\xi}[y -y_h(u),y -\Pi_n\left(\mathcal{R}_h(y)\right)] + \underbrace{a_{\xi}[y -y_h(u),\Pi_n\left(\mathcal{R}_h(y)\right) -y_h(u)]}_{=0} \\
		& \quad \leq c_{ct}\lVert y-y_h(u) \rVert _{\xi} \lVert  y -\Pi_n\left(\mathcal{R}_h(y)\right) \rVert _{\xi}.
	\end{align*}
	Then, by the approximation property \eqref{ineq:main_appr}, we get
	\begin{align*}
		\lVert y -y_h(u) \rVert _{\xi} 	&\leq \frac{c_{ct}}{c_{cv}}\lVert  y -\Pi_n\left(\mathcal{R}_h(y)\right) \rVert _{\xi}\\
		& \leq Ch \lVert y \rVert_{L^2(H^2(\mathcal{D});\Gamma)}
		+ \sum_{n=1}^{N}\bigg(\frac{k_n}{2}\bigg)^{q_n+1}\dfrac{\lVert \partial^{q_n+1}_{\xi_n} y \rVert_{L^2(H^1(\mathcal{D});\Gamma)} }{(q_n+1)!},
	\end{align*}
	which is the desired result \eqref{ineq:y_hy}. Analogously, we deduce that
	\begin{align*}
		& c_{cv}\lVert p -p_h(u) \rVert _{\xi}^2	\nonumber \\
		& \quad \leq a_{\xi}[p -p_h(u),p- p_h(u)] \\
		& \quad \leq a_{\xi}[p -\Pi_n\left(\mathcal{R}_h(y)\right),p -p_h(u)] + \underbrace{a_{\xi}[\Pi_n\left(\mathcal{R}_h(y)\right) -p_h(u),p -p_h(u)]}_{=0} \\
		& \quad = (1+\gamma) [y-y_h(u),p -\Pi_n\left(\mathcal{R}_h(y)\right)]_{\xi} +  \gamma \big[\widetilde{\mathbb{E}}[y_h(u)-y] ,p -\Pi_n\left(\mathcal{R}_h(y)\right) \big]_{\xi} \\
		& \quad \leq (1+2\gamma)  \lVert y-y_h(u) \rVert _{\xi} \lVert  p -\Pi_n\left(\mathcal{R}_h(y)\right) \rVert _{L^2(L^2(\mathcal{D});\Gamma)}\\
		& \quad \leq \frac{(1+2\gamma)}{2}  \lVert y-y_h(u) \rVert^2_{\xi} + \frac{(1+2\gamma)}{2}   \lVert  p -\Pi_n\left(\mathcal{R}_h(y)\right) \rVert^2 _{L^2(L^2(\mathcal{D});\Gamma)},
	\end{align*}
	where the definition of bilinear forms, the procedure applied in \eqref{ineq:negGamma}, and Young's inequality are used. Then, using the approximation property \eqref{ineq:main_appr} and \eqref{ineq:y_hy}, we complete the proof of \eqref{ineq:p_hp}.
\end{proof}

Now, we finalize the error analysis   by combining the findings in Lemmas~\ref{lemma:control} and~\ref{lemma:state_adjo}.
\begin{theorem}\label{thm:error}
	Assume that  $ (y,u,p) $ and $ (y_h,u_h,p_h) $, respectively, are the solutions of \eqref{eqn:opti1} and \eqref{eqn:opti_disc}. Then, it holds that
	\begin{align}\label{ineq:error}
		&	\lVert u-u_h \rVert _{L^2(L^2(\mathcal{D});\Gamma)} + \lVert y -y_h \rVert _{\xi} + \lVert p -p_h \rVert _{\xi}  \nonumber \\
		& \; \leq C h^{3/2} \|u\|_{L^2(W^{1,\infty}(\mathcal{D});\Gamma)} + Ch \big( \lVert y \rVert_{L^2(H^2(\mathcal{D});\Gamma)} +\lVert p \rVert_{L^2(H^2(\mathcal{D});\Gamma)}\big) \nonumber \\
		& \;\;\; + C \sum_{n=1}^{N}\bigg(\frac{k_n}{2}\bigg)^{q_n+1}\dfrac{\Big(\lVert \partial^{q_n+1}_{\xi_n} y \rVert_{L^2(H^1(\mathcal{D});\Gamma)} + \lVert \partial^{q_n+1}_{\xi_n} p \rVert_{L^2(H^1(\mathcal{D});\Gamma)}\Big)}{(q_n+1)!}.
	\end{align}
\end{theorem}

\begin{proof}
	From \eqref{ineq:control} and  \eqref{ineq:p_hp}, we obtain that
	\begin{align}\label{ineq:thm1}
		&\lVert u-u_h \rVert _{L^2(L^2(\mathcal{D});\Gamma)} \nonumber \\
		& \quad \leq C h^{3/2} \|u\|_{L^2(W^{1,\infty}(\mathcal{D});\Gamma)} + Ch \big( \lVert y \rVert_{L^2(H^2(\mathcal{D});\Gamma)} +\lVert p \rVert_{L^2(H^2(\mathcal{D});\Gamma)}\big) \nonumber \\
		& \qquad + C \sum_{n=1}^{N}\bigg(\frac{k_n}{2}\bigg)^{q_n+1}\dfrac{\Big(\lVert \partial^{q_n+1}_{\xi_n} y \rVert_{L^2(H^1(\mathcal{D});\Gamma)} + \lVert \partial^{q_n+1}_{\xi_n} p \rVert_{L^2(H^1(\mathcal{D});\Gamma)}\Big)}{(q_n+1)!}.
	\end{align}
	Moreover, by Lemmas~\ref{lemma:y_hp_h} and~\ref{lemma:state_adjo}, and  the bound \eqref{ineq:thm1}, we obtain
	\begin{align}\label{ineq:thm2}
		& \lVert y -y_h \rVert _{\xi} + \lVert p -p_h \rVert _{\xi} \nonumber \\
		& \;\leq \lVert y -y_h(u) \rVert _{\xi} +   \lVert y_h(u)-y_h \rVert _{\xi}  + \lVert p -p_h(u) \rVert _{\xi}  +\lVert  p_h(u)-p_h \rVert _{\xi}\nonumber \\
		& \; \leq  C \lVert u-u_h \rVert _{L^2(L^2(\mathcal{D});\Gamma)} + Ch \big( \lVert y \rVert_{L^2(H^2(\mathcal{D});\Gamma)} + \lVert p \rVert_{L^2(H^2(\mathcal{D});\Gamma)} \big) \nonumber \\
		&\quad \;+ C \sum_{n=1}^{N}\bigg(\frac{k_n}{2}\bigg)^{q_n+1} \left( \dfrac{\lVert \partial^{q_n+1}_{\xi_n} y \rVert_{L^2(H^1(\mathcal{D});\Gamma)}+ \lVert \partial^{q_n+1}_{\xi_n} q \rVert_{L^2(H^1(\mathcal{D});\Gamma) } }{(q_n+1)!} \right).
	\end{align}
	Thus, by combining \eqref{ineq:thm1} and \eqref{ineq:thm2}, we deduce the desired result \eqref{ineq:error}.
\end{proof}


\section{Matrix Formulation}\label{sec:matrix}

In this section, we first construct the matrix formulation of the underlying problem \eqref{eqn:objec2}--\eqref{eqn:state2} by employing the \textquotedblleft optimize-then-discretize\textquotedblright \, approach; see, e.g., \cite{FTroeltzsch_2010a}. In this methodology, one first obtains the optimality system  \eqref{eqn:opti1} of the infinite-dimensional optimization problem, and then discretizes the optimality system by a stochastic discontinuous Galerkin method discussed in Section~\ref{sec:finite}. Later, we propose a low--rank variant of generalized minimal residual (GMRES) method with a suitable  preconditioner  to solve the corresponding  linear system.

\subsection{State system}

After an application of the discretization techniques discussed in Section~\ref{sec:finite}, one gets the following linear system for the state part of the optimality system \eqref{eqn:opti1}:
\begin{equation}\label{tensormtrx}
	\underbrace{\left( \sum_{i=0}^{N} \mathcal{G}_i \otimes \mathcal{K}_i \right)}_{\mathcal{A}} \, \mathbf{y} - \underbrace{\left(\mathcal{G}_0\otimes M\right)}_{\mathcal{M}} \mathbf{u}= \underbrace{\left( \sum_{i=0}^{N} \mathbf{g}_i \otimes \mathbf{f}_i\right)}_{\mathcal{F}},
\end{equation}
where
$
\mathbf{y} =\left( y_0, \ldots,y_{J-1} \right)^T \;\; \hbox{and} \; \; \mathbf{u} =\left( u_0, \ldots,u_{J-1} \right)^T \;\; \hbox{with} \;\;  y_i,u_i \in \mathbb{R}^{N_d}, i=0,1,\ldots,J-1
$
and $N_d$ corresponds to the degree of freedom for the spatial discretization. The mass matrix $ M\in \mathbb{R}^{N_{d} \times N_d} $, the stiffness matrices $\mathcal{K}_i \in \mathbb{R}^{N_{d} \times N_d}$,  and the right--hand side vectors $\mathbf{f}_i \in \mathbb{R}^{N_d}$  are given, respectively, by
\begin{eqnarray*}
	M(r,s)\hspace{-1mm}&=&\hspace{-3mm}  \sum \limits_{K \in \mathcal{T}_h} \int \limits_{K} \varphi_{r} \varphi_{s}  \, d\boldsymbol{x},  \\
	\mathcal{K}_0(r,s)\hspace{-1mm}&=&\hspace{-3mm}  \sum \limits_{K \in \mathcal{T}_h} \int \limits_{K} \left(  \overline{a} \, \nabla \varphi_{r} \cdot  \nabla \varphi_{s} + \overline{\mathbf{b}} \cdot \nabla \varphi_{r} \varphi_{s} \right) \, d\boldsymbol{x}  \\
	&& - \hspace{-2.5mm} \sum \limits_{E  \in \mathcal{E}^{0}_h \cup \mathcal{E}_h^{\partial}} \int \limits_E \big(  \average{\overline{a} \, \nabla \varphi_{r}}  \jump{\varphi_{s}} +  \average{\overline{a} \, \nabla \varphi_{s}}  \jump{\varphi_{r}} \big) \, ds \\
	&& + \hspace{-2.5mm} \sum \limits_{ E \in \mathcal{E}^{0}_h \cup \mathcal{E}_h^{\partial}} \frac{\sigma}{h_E} \int \limits_E  \jump{\varphi_{r}} \cdot \jump{\varphi_{s}} \, ds +  \hspace{-2.5mm}  \sum \limits_{K \in \mathcal{T}_h}\; \int \limits_{\partial K^{-} \backslash \partial \mathcal{D}}\hspace{-3mm} \overline{\mathbf{b}} \cdot \mathbf{n}_E (\varphi_{r}^e -\varphi_{r})\varphi_{s} \, ds\\
	&&	- \sum \limits_{K  \in \mathcal{T}_h} \; \int \limits_{\partial K^{-} \cap \partial \mathcal{D}^{-}}   \hspace{-3mm}\overline{\mathbf{b}} \cdot \mathbf{n}_E \varphi_{r} \varphi_{s}  \, ds, \\	
	\mathcal{K}_i(r,s) \hspace{-1mm}&=& \hspace{-3mm} \sum \limits_{K \in \mathcal{T}_h} \int \limits_{K} \left( \Big(\kappa_a \sqrt{\lambda_i^{a}}\phi_{i}^{a}\Big)  \nabla \varphi_{r} \cdot  \nabla \varphi_{s} +  \Big(\kappa_{\mathbf{b}} \sqrt{\lambda_i^{\mathbf{b}}}\phi_{i}^{\mathbf{b}}\Big)  \cdot \nabla \varphi_{r} \varphi_{s} \right) \, d\boldsymbol{x}  \\
	&&\hspace{-3mm} -\hspace{-3.5mm}	  \sum \limits_{ E \in \mathcal{E}^{0}_h \cup \mathcal{E}_h^{\partial}} \int \limits_E \left (\average{\Big(\kappa_{a} \sqrt{\lambda_i^{a}}\phi_{i}^{a}\Big)  \nabla \varphi_{r}}  \jump{\varphi_{s}} + \hspace{-1.2mm}	\average{\Big(\kappa_{a} \sqrt{\lambda_i^{a}}\phi_{i}^{a}\Big)  \nabla \varphi_{s}}  \jump{\varphi_{r}} \right) \, ds \\
	&& \hspace{-3mm}+ \sum \limits_{ E \in \mathcal{E}^{0}_h \cup \mathcal{E}_h^{\partial}} \frac{\sigma}{h_E} \int \limits_E  \jump{\varphi_{r}} \cdot \jump{\varphi_{s}} \, ds \\
	&&	\hspace{-3mm} + \sum \limits_{K \in \mathcal{T}_h} \int \limits_{\partial K^{-} \backslash \partial \mathcal{D}} \hspace{-4mm} \Big(\kappa_{\mathbf{b}} \sqrt{\lambda_i^{\mathbf{b}}}\phi_{i}^{\mathbf{b}}\Big)  \cdot \mathbf{n}_E (\varphi_{r}^e-\varphi_{r})\varphi_{s} \, ds \\
	&&	\hspace{-3mm} - \sum \limits_{T \in \mathcal{T}_h}  \int \limits_{\partial K^{-} \cap \partial \mathcal{D}^{-}} \hspace{-4mm}\Big(\kappa_{\mathbf{b}} \sqrt{\lambda_i^{\mathbf{b}}}\phi_{i}^{\mathbf{b}}\Big)  \cdot \mathbf{n}_E \varphi_{r} \varphi_{s} \, ds, \\
	f_0(s)\hspace{-1mm}&=&\hspace{-3mm}  \sum \limits_{K \in \mathcal{T}_h} \int \limits_{K} f\varphi_{s} \; d\boldsymbol{x}
	+ \sum \limits_{E \in \mathcal{E}_h^{\partial}}\frac{\sigma}{h_E} \int \limits_E  y_{DB}  \jump{\varphi_{s}} \; ds
	-\sum \limits_{E \in \mathcal{E}_h^{\partial}}  \int \limits_E y_{DB} \average{\overline{a} \nabla \varphi_{s}} \; ds \\
	&& \hspace{-3mm}- \sum \limits_{K \in \mathcal{T}_h}  \int \limits_{\partial K^{-} \cap \partial \mathcal{D}^{-}} \overline{\mathbf{b}} \cdot \mathbf{n}_E \, y_{DB} \,\varphi_{s}  \; ds, \\
	f_i(s)\hspace{-1mm}&=&\hspace{-3mm} \sum \limits_{E \in \mathcal{E}_h^{\partial}}\frac{\sigma }{h_E} \int \limits_E  y_{DB} \jump{\varphi_{s}} \; ds
	-\sum \limits_{E \in \mathcal{E}_h^{\partial}}  \int \limits_E y_{DB} \average{\Big(\kappa_{a} \sqrt{\lambda_i^{a}}\phi_{i}^{a}\Big) \nabla \varphi_{s}} \; ds  \;  \\
	\\
	&&\hspace{-3mm} - \sum \limits_{K \in \mathcal{T}_h} \; \int \limits_{\partial K^{-} \cap \partial \mathcal{D}^{-}} \Big(\kappa_{\mathbf{b}} \sqrt{\lambda_i^{\mathbf{b}}}\phi_{i}^{\mathbf{b}}\Big) \cdot \mathbf{n}_E \, y_{DB} \, \varphi_{s} \; ds,
\end{eqnarray*}
where $\{\varphi_i(\boldsymbol{x} )\}$ corresponds to the set of basis functions for the spatial discretization, i.e., $V_h = \hbox{span} \{\varphi_i(\boldsymbol{x} )\}$.

On the other hand, for $i=0,\ldots, N$,  the stochastic matrices $\mathcal{G}_i \in \mathbb{R}^{J \times J}$ and  the stochastic  vectors $\mathbf{g}_i \in \mathbb{R}^{J}$ in \eqref{tensormtrx} are given, respectively, by
\begin{subequations}\label{stocmtrx}
	\begin{align}
		& \mathcal{G}_0(r,s) = \ang{\psi_r \psi_s}, \qquad  \mathcal{G}_i(r,s)= \ang{\xi_i \psi_r \psi_s},  \\
		& \mathbf{g}_0(r)= \ang{\psi_r}, \qquad \qquad \mathbf{g}_i(r)= \ang{\xi_i \psi_r}.
	\end{align}
\end{subequations}
In \eqref{stocmtrx}, each stochastic basis function $\psi_i(\xi)$ is  a product of $N$ univariate orthogonal polynomials, i.e.,
$
\psi_i(\xi) = \psi_{i_1}(\xi) \psi_{i_2}(\xi) \ldots \psi_{i_N}(\xi),
$
where the multi--index $i$ is defined by $i=(i_1, i_2, \ldots, i_N)$ with $\sum \limits_{s=1}^N i_s \leq Q$. In this paper, Legendre polynomials are chosen as the stochastic basis functions since the underlying random variables have  uniform distribution on $[-\sqrt{3}, \sqrt{3}]$ \cite{PCiloglu_HYucel_2022}. Then, $\mathcal{G}_0$ becomes an identity matrix, whereas $\mathcal{G}_k, \; k>0$,  contains at most two nonzero entries per row; see, e.g., \cite{OGErnst_EUllmann_2010,CEPowell_HCElman_2009}. On the other hand, $\textbf{g}_i$ is the first column of $\mathcal{G}_i, \; i=0,1,\ldots,N$.

\subsection{Matrix formulation of the optimality system}

The discrete optimality system in \eqref{eqn:opti_disc} can be represented as a block matrix system including the state, adjoint, and variational equations in the finite dimensional setting. To solve the underlying block linear system, \textquotedblleft the primal--dual active set (PDAS) methodology  as a semi-smooth Newton step\textquotedblright \, is applied; see, e.g., \cite{MBergounioux_KIto_KKunisch_1999a} for more details. After a definition of the active sets
\begin{eqnarray*}
	A^{-} &=&\bigcup \{\boldsymbol{x} \in K \; : \; -\mathbf{p} - \mu \mathbf{u}_a < 0, \; \forall K \in \mathcal{T}_h  \},\\
	A^{+} &=&\bigcup \{\boldsymbol{x} \in K \; : \;-\mathbf{p} - \mu \mathbf{u}_b > 0, \; \forall K \in \mathcal{T}_h   \},
\end{eqnarray*}
and the inactive set
\begin{eqnarray*}
	\mathcal{I} &=& \mathcal{T}_h \setminus \left( A^{-} \cup  A^{+}\right),
\end{eqnarray*}
the block formulation becomes
\begin{subequations}\label{eqn:Cons_opti_matrix}
	\begin{align}
		& \mathcal{A} \mathbf{y} - \mathcal{M}_I \mathbf{u} = \mathcal{F}, \, \label{eqn:Cons_Mdisc}\\
		& \mathcal{A}^* \mathbf{p} - \mathcal{M}_{\gamma} \mathbf{y} = -\mathcal{F}^d,\label{eqn:Cons_MAdjdisc}\\
		&\hspace{-3mm} \left( \mathcal{G}_0 \otimes \diag(\mathbbm{1}_{\mathcal{I}})\right)  \mathbf{p} + \mu  \left( \mathcal{G}_0 \otimes I \right) \mathbf{u} = \left( \mathbf{g}_0 \otimes \mathbbm{1}_{A^{-}}\right) \mu\, \mathbf{u}_a + \left( \mathbf{g}_0 \otimes \mathbbm{1}_{A^{+}}\right) \mu \,\mathbf{u}_b ,\label{eqn:Cons_opti_Mdisc3}
	\end{align}
\end{subequations}
where
\begin{align*}
	\mathcal{M}_I &:= I \otimes M, \\
    \mathcal{F}^d &:= \mathbf{g}_0 \otimes \mathbf{y}^d  \hbox{      with    }  \mathbf{y}^d (s) = \sum \limits_{K \in \mathcal{T}_h} \int \limits_{K} y^d \varphi_{s} \; d\boldsymbol{x}, \\
	\mathcal{M}_{\gamma} &:= \left(\mathcal{G}_0 \otimes M\right) + \gamma\left(\mathcal{M}_0 \otimes M\right) \hbox{   with } \mathcal{M}_0  = \hbox{diag} \left( 0, \ang{ \psi_1 }^2, \ldots, \ang{ \psi_{J-1} }^2 \right),
\end{align*}
and  $\mathbbm{1}_{A^{-}}$, $\mathbbm{1}_{A^{+}}$, and $ \mathbbm{1}_{\mathcal{I}}$  correspond to the characteristic functions of $ A^{-} $, $ A^{+} $, and $\mathcal{I}$, respectively. Equivalently,  $\mathcal{M}_{\gamma}$ can be rewritten as
\[
\mathcal{M}_{\gamma} :=  \mathcal{G}_{\gamma}\otimes M, \quad \hbox{with} \quad  \mathcal{G}_{\gamma} :=\mathcal{G}_0 + \gamma\mathcal{M}_0,
\]
where
\begin{align}
	\mathcal{G}_{\gamma}(r,s) =
	\begin{cases}
		\ang{ \psi_0 }^2, &\mbox{if } r=s=0, \\
		\left(1+\gamma\right)\ang{ \psi_r }^2, &\mbox{if } r=s=1,\ldots,J-1, \\
		0, & \mbox{otherwise}.
	\end{cases}
\end{align}
Rearranging \eqref{eqn:Cons_opti_matrix} gives us  the following linear matrix system
\begin{align}\label{eqn:Cons_MatrixSystem}
	\left[ \begin{matrix}
		\mathcal{M}_{\gamma} & 0                & -\mathcal{A}^* \\
		0 & \mu \left( \mathcal{G}_0 \otimes I \right)  &  \mathcal{G}_0 \otimes \diag(\mathbbm{1}_{\mathcal{I}}) \\
		-\mathcal{A} & \mathcal{M}_I      & 0
	\end{matrix}\right]\hspace{-1.5mm}
	\left[ \begin{matrix}
		\mathbf{y}\\
		\mathbf{u}\\
		\mathbf{p}
	\end{matrix}\right]
    \hspace{-1.5mm}	= \hspace{-1.5mm}
	\left[ \begin{matrix}
		\mathcal{F}^d\\
		\mu \big(\left( \mathbf{g}_0 \otimes \mathbbm{1}_{A^{-}}\right)  \mathbf{u}_a + \left( \mathbf{g}_0 \otimes \mathbbm{1}_{A^{+}}\right)  \mathbf{u}_b\big) \\
		-\mathcal{F}
	\end{matrix}\right]\hspace{-1.3mm},
\end{align}
which is a saddle point system. We note that since Legendre polynomials are used, $\mathcal{G}_0 = I$, and hence, $\mathcal{M}_I =\mathcal{M} $.

In practical implementations, the saddle point system \eqref{eqn:Cons_MatrixSystem} typically becomes very large, depending on the length  of the random vector $\xi$ and the number of refinement in the spatial discretization. We break this curse of dimensionality  by using a low--rank  approximation, which reduces both the computational complexity and memory requirements by using  a Kronecker--product structure of the matrices  defined in \eqref{eqn:Cons_MatrixSystem}.

\subsection{Low-rank approach}\label{sec:lowrank}

We first introduce the notation and   some  basic properties of the low--rank approach. Let $\Theta=[\theta_1,\ldots,\theta_J]\in \mathbb{R}^{N_d \times J}$ and  let the operators \texttt{vec}($ \cdot $) and \texttt{mat}($ \cdot $) be  isomorphic mappings between $\mathbb{R}^{N_d \times J}$  and  $\mathbb{R}^{N_dJ}$ as following
\[
\texttt{vec}:\mathbb{R}^{N_d \times J} \rightarrow \mathbb{R}^{N_dJ},  \qquad \texttt{mat}: \mathbb{R}^{N_dJ} \rightarrow \mathbb{R}^{N_d\times J},
\]
where $N_d$ and $J$ are the degrees of freedom for the spatial discretization and the total degree of the multivariate stochastic basis polynomials, respectively. The matrix inner product is defined by $\ang{U,V}_F=\text{trace}(U^TV)$ with $\|U\|_F = \sqrt{\ang{U,V}_F}$.  Further, the following relation holds, see, e.g., \cite{DKressner_CTobler_2011}:
\begin{eqnarray}\label{ide:vec}
	\texttt{vec}(U \Theta V)=(V^T\otimes U)\texttt{vec}(\Theta).
\end{eqnarray}
Now, we can interpret  the system \eqref{eqn:Cons_MatrixSystem}  as follows
\begin{align}\label{eqn:MatrixSystem2}
	\underbrace{ \left[ \begin{matrix}
			\mathcal{G}_{\gamma}\otimes M & 0 & -\sum\limits_{i=0}^{N} \mathcal{G}_i \otimes \mathcal{K}_i^* \\
			0 & \mu  \left(\mathcal{G}_0\otimes I\right)  & \mathcal{G}_0\otimes \diag(\mathbbm{1}_{\mathcal{I}})	\\
			-\sum\limits_{i=0}^{N} \mathcal{G}_i \otimes \mathcal{K}_i & \mathcal{G}_0\otimes M  & 0
		\end{matrix}\right]  }_{\mathcal{L}}
	\underbrace{	\left[ \begin{matrix}
			\texttt{vec}(Y)\\
			\texttt{vec}(U)\\
			\texttt{vec}(P)
		\end{matrix}\right]}_{\Theta}
	=
	\underbrace{\left[ \begin{matrix}
			\texttt{vec}(B_1)\\
			\texttt{vec}(B_2)\\
			\texttt{vec}(B_3)
		\end{matrix}\right]}_{\mathcal{B}},
\end{align}
where
\begin{align*}
	& Y= \left(y_0,\ldots,y_{J-1}\right), \;  U= \left(u_0,\ldots,u_{J-1}\right), \; P= \left(p_0,\ldots,p_{J-1}\right), \;  B_1 =\texttt{mat}\left( \mathcal{F}^d\right), \\
	&  B_2 =\texttt{mat}\left(\mu \big(\left( \mathbf{g}_0 \otimes \mathbbm{1}_{A^{-}}\right) \mathbf{u}_a + \left( \mathbf{g}_0 \otimes \mathbbm{1}_{A^{+}}\right)  \mathbf{u}_b\big)\right), \; B_3 = \texttt{mat}\left( -\mathcal{F}\right).
\end{align*}
By the identity \eqref{ide:vec}, we have
\begin{align}\label{eqn:AX}
	\mathcal{L}\Theta = \texttt{vec} \left(
	\left[ \begin{matrix}
		MY\mathcal{G}_{\gamma}^T - \sum\limits_{i=0}^{N} \mathcal{K}_i^*P\mathcal{G}_i^T  \\
		\mu IU\mathcal{G}_0^T + \diag(\mathbbm{1}_{\mathcal{I}})P\mathcal{G}_0^T\\
		- \sum\limits_{i=0}^{N} \mathcal{K}_i Y \mathcal{G}_i^T + MU\mathcal{G}_0^T
	\end{matrix}\right]
	\right) = \texttt{vec} \left(
	\left[ \begin{matrix}
		B_1  \\
		B_2\\
		B_3
	\end{matrix}\right]
	\right).
\end{align}
Assuming that the  matrices $\Theta$  and  $\mathcal{B}$ have the following  low--rank representations, see, e.g., \cite{MStoll_TBreiten_2015,PBenner_TBreiten_2013,MAFreitag_DLHGreen_2018},
\begin{align}\label{eqn:lowRank}
	Y &= W_YV_Y^T, \qquad W_Y \in \mathbb{R}^{N_d \times r_Y}, \; V_Y \in \mathbb{R}^{J \times r_Y},\nonumber\\
	U &= W_UV_U^T, \qquad W_U \in \mathbb{R}^{N_d \times r_U}, \; V_U \in \mathbb{R}^{J \times r_U},\\
	P &= W_PV_P^T, \qquad W_P \in \mathbb{R}^{N_d \times r_P}, \; V_P \in \mathbb{R}^{J \times r_P},\nonumber\\
	B_1 &= B_{11} B_{12}^T \qquad B_{11} \in \mathbb{R}^{N_d \times r_{B_1}}, \; B_{12} \in \mathbb{R}^{J \times r_{B_1}},\nonumber\\
	B_2 &= B_{21} B_{22}^T \qquad B_{21} \in \mathbb{R}^{N_d \times r_{B_2}}, \; B_{22} \in \mathbb{R}^{J \times r_{B_2}},\nonumber\\
	B_3 &= B_{31} B_{32}^T \qquad B_{31} \in \mathbb{R}^{N_d \times r_{B_3}}, \; B_{32} \in \mathbb{R}^{J \times r_{B_3}},\nonumber
\end{align}
with $r_Y,r_U, r_P, r_{B_1}, r_{B_2}, r_{B_3} \ll N_d, J$,  \eqref{eqn:AX} can be stated as follows
\begin{align}\label{eqn:AX2}
	\left[ \begin{matrix}
		MW_YV_Y^T\mathcal{G}_{\gamma}^T - \sum\limits_{i=0}^{N} \mathcal{K}_i^*W_PV_P^T\mathcal{G}_i^T  \\
		\mu IW_UV_U^T\mathcal{G}_0^T + \diag(\mathbbm{1}_{\mathcal{I}}) W_PV_P^T\mathcal{G}_0^T\\
		- \sum\limits_{i=0}^{N} \mathcal{K}_i W_YV_Y^T \mathcal{G}_i^T + MW_UV_U^T\mathcal{G}_0^T
	\end{matrix}\right] =
	\left[ \begin{matrix}
		B_{11} B_{12}^T \\
		B_{21} B_{22}^T\\
		B_{31} B_{32}^T
	\end{matrix}\right],
\end{align}
where \texttt{vec} operator is ignored. Moreover, the three block rows in \eqref{eqn:AX2} can be written as
\begin{subequations}\label{eqn:AX3}
	\begin{align}
		&\underbrace{\left[ \begin{matrix}
				MW_Y  &  - \sum\limits_{i=0}^{N} \mathcal{K}_i^*W_P	\end{matrix}\right] }_{\widehat{W}_1}
		\underbrace{\left[ \begin{matrix}
				\mathcal{G}_{\gamma}V_Y &  \mathcal{G}_iV_P	\end{matrix}\right]^T}_{\widehat{V}_1^T},\\
		&\underbrace{\left[ \begin{matrix}
				\mu I W_U  &  \diag(\mathbbm{1}_{\mathcal{I}}) W_P	\end{matrix}\right] }_{\widehat{W}_2}
		\underbrace{\left[ \begin{matrix}
				\mathcal{G}_0 V_U &  \mathcal{G}_0 V_P	\end{matrix}\right]^T}_{\widehat{V}_2^T},\\
		&\underbrace{\left[ \begin{matrix}
				- \sum\limits_{i=0}^{N} \mathcal{K}_i W_Y  &  M W_U	\end{matrix}\right] }_{\widehat{W}_3}
		\underbrace{\left[ \begin{matrix}
				\mathcal{G}_i V_Y & \mathcal{G}_0 V_U	\end{matrix}\right]^T}_{\widehat{V}_3^T},
	\end{align}
\end{subequations}
in low--rank formats $\widehat{W}_i\widehat{V}_i^T \; \hbox{for} \; i=1,2,3$. By the usage of  \eqref{eqn:AX3}, the low--rank approximate solutions to \eqref{eqn:MatrixSystem2} can be obtained; see  Algorithm~\ref{alg:GMRES} modified from \cite{MAFreitag_DLHGreen_2018} for details of the low--rank implementation of GMRES. Moreover, with the help of the following fact
\[
\text{trace}(A^TB) = \texttt{vec}(A)^T \texttt{vec}(B),
\]
the inner products $\ang{A,B}_F=\text{trace}(A^TB)$ in the iterative low--rank algorithm  can be computed efficiently. For instance,  the inner product computation in Algorithm~\ref{alg:GMRES} denoted by
\[
\texttt{trprod}(A_{11}, A_{12}, A_{21}, A_{22}, A_{31}, A_{32}, B_{11}, B_{12}, B_{21}, B_{22}, B_{31}, B_{32})
\]
can be computed as following
\begin{eqnarray*}
	\ang{A,B}_F &=& \text{trace}\left( \left(A_{11}A_{12}^T\right)^T \left(B_{11}B_{12}^T\right)^T\right) + \text{trace}\left( \left(A_{21}A_{22}^T\right)^T \left(B_{21}B_{22}^T\right)^T\right) \nonumber\\
	&& + \text{trace}\left( \left(A_{31}A_{32}^T\right)^T \left(B_{31}B_{32}^T\right)^T\right) \\
	&=& \text{trace}\left( A_{11}^TB_{11}A_{12}^TB_{12}\right) +\text{trace}\left( A_{21}^TB_{21}A_{22}^TB_{22}\right)\nonumber \\
	&& + \text{trace}\left( A_{31}^TB_{31}A_{32}^TB_{32}\right),
\end{eqnarray*}
where
\[
A = \texttt{vec} \left( \left[ \begin{matrix}
	A_{11} A_{12}^T\\
	A_{21} A_{22}^T\\
	A_{31} A_{32}^T
\end{matrix}\right]\right), \qquad
B = \texttt{vec} \left( \left[ \begin{matrix}
	B_{11} B_{12}^T\\
	B_{21} B_{22}^T\\
	B_{31} B_{32}^T
\end{matrix}\right]\right).
\]

\begin{algorithm}[htp!]
	\footnotesize
	\caption{Low--rank  preconditioned GMRES (LRPGMRES)}
	\label{alg:GMRES}
	\hspace*{\algorithmicindent}{\textbf{Input:} Coefficient matrix $\mathcal{L} : \mathbb{R}^{3N_d\times J} \rightarrow \mathbb{R}^{3N_d\times J} $, inverse of the preconditioner matrix  $\mathcal{P}^{-1}_0: \mathbb{R}^{3N_d\times J} \rightarrow \mathbb{R}^{3N_d\times J}$, and right--hand side matrix $\mathcal{B}$ in the low--rank formats. Truncation operator $\mathcal{T}$ with given tolerance $\epsilon_{trunc}$.}\\
	\hspace*{\algorithmicindent}{\textbf{Output:} Matrix $\Theta \in \mathbb{R}^{3N_d\times J}$ satisfying $\lVert \mathcal{L}(\Theta)-\mathcal{B} \rVert_F / \|\mathcal{B}\|_F \leq \epsilon_{tol}$.}\\
	\vspace{-3mm}
	\begin{algorithmic}[1]
		\State{Choose initial guess $\Theta_{11}^{(0)}, \Theta_{12}^{(0)}, \Theta_{21}^{(0)}, \Theta_{22}^{(0)}, \Theta_{31}^{(0)}, \Theta_{33}^{(0)}$.}
		\State{$(\widetilde{\Theta}_{11}, \widetilde{\Theta}_{12}, \widetilde{\Theta}_{21}, \widetilde{\Theta}_{22}, \widetilde{\Theta}_{31}, \widetilde{\Theta}_{32}) = \mathcal{L}(\Theta_{11}^{(0)}, \Theta_{12}^{(0)}, \Theta_{21}^{(0)}, \Theta_{22}^{(0)}, \Theta_{31}^{(0)}, \Theta_{32}^{(0)})$. \hfill $\widetilde{\Theta}_{ij}\leftarrow \mathcal{T}(\widetilde{\Theta}_{ij}) $}
		\State{$R_{11}^{(0)} = \{B_{11}, \;\; -\Theta_{11}^{(0)}\} $, \quad  $R_{12}^{(0)} = \{B_{12}, \;\; \Theta_{12}^{(0)}\}. $}
		\State{$R_{21}^{(0)} = \{B_{21}, \;\; -\Theta_{21}^{(0)}\} $, \quad  $R_{22}^{(0)} = \{B_{22}, \;\; \Theta_{22}^{(0)}\}. $ \hfill $R_{ij}^{(0)}\leftarrow \mathcal{T}(R_{ij}^{(0)})$ }
		\State{$R_{31}^{(0)} = \{B_{31}, \;\; -\Theta_{31}^{(0)}\} $, \quad  $R_{32}^{(0)} = \{B_{32}, \;\; \Theta_{32}^{(0)}\}. $}
		\State{$ \lVert R^0\rVert = \sqrt{\hbox{trprod}(R_{11}^{(0)},\ldots,R_{11}^{(0)},\ldots)}.$}	
		\State{$V_{11}^{(0)} = R_{11}^{(0)}/ \lVert R^0\rVert_F $, \quad  $V_{12}^{(0)} = R_{12}^{(0)}.$}		
		\State{$V_{21}^{(0)} = R_{21}^{(0)}/ \lVert R^0\rVert_F $, \quad  $V_{22}^{(0)} = R_{22}^{(0)}.$ \hfill $V_{ij}^{(0)}\leftarrow \mathcal{T}(V_{ij}^{(0)})$}	
		\State{$V_{31}^{(0)} = R_{31}^{(0)}/ \lVert R^0\rVert_F $, \quad  $V_{32}^{(0)} = R_{32}^{(0)}.$}	
		\State{$\gamma=[\gamma_1,0,\ldots,0]$, \qquad  $\gamma_1=\sqrt{\hbox{trprod}(V_{11}^{(0)}, \ldots, V_{11}^{(0)},\ldots)}.$}
		\While{$ i \leq \text{maxit}$}
		\State{$(Z_{11}^{(i)}, Z_{12}^{(i)}, Z_{21}^{(i)}, Z_{22}^{(i)}, Z_{31}^{(i)}, Z_{32}^{(i)})\hspace{-1mm}=\hspace{-1mm}\mathcal{P}^{-1}_0(V_{11}^{(i)}, V_{12}^{(i)}, V_{21}^{(i)}, V_{22}^{(i)}, V_{31}^{(i)}, V_{32}^{(i)})$, \hfill $Z_{ij}^{(i)}\leftarrow \mathcal{T}(Z_{ij}^{(i)}) $}
		\State{$(W_{11}, W_{12}, W_{21}, W_{22}, W_{31}, W_{32}) = \mathcal{L}(Z_{11}^{(i)}, Z_{12}^{(i)}, Z_{21}^{(i)}, Z_{22}^{(i)}, Z_{31}^{(i)}, Z_{32}^{(i)})$. \hfill $W_{ij}\leftarrow \mathcal{T}(W_{ij}) $}
		\For{$ j=1,\ldots,i$}
		\State{$m_{j,i}=\sqrt{\hbox{trprod}(W_{11}, \ldots, V_{11}^{(j)},\ldots)}$}
		\State{$W_{11} = \{W_{11}, \;\; -m_{j,i} V_{11}^{(i)}\} $, \quad  $W_{12} = \{W_{12}, \;\; V_{12}^{(j)}\}. $}
		\State{$W_{21} = \{W_{21}, \;\; -m_{j,i} V_{21}^{(i)}\} $, \quad  $W_{22} = \{W_{22}, \;\; V_{22}^{(j)}\}. $ \hfill $W_{ij}\leftarrow \mathcal{T}(W_{ij}) $}
		\State{$W_{31} = \{W_{31}, \;\; -m_{j,i} V_{31}^{(i)}\} $, \quad  $W_{32} = \{W_{32}, \;\; V_{32}^{(j)}\}. $}
		\EndFor
		\State{$m_{i+1,i}=\sqrt{\hbox{trprod}(W_{11}, \ldots, W_{11},\ldots)}$}
		\State{$V_{11}^{(i+1)} = W_{11}/m_{i+1,k} $, \quad  $V_{12}^{(k+1)} = W_{12}.$}	
		\State{$V_{21}^{(i+1)} = W_{21}/m_{i+1,k} $, \quad  $V_{22}^{(k+1)} = W_{22}.$ \hfill $V_{ij}^{(i+1)}\leftarrow \mathcal{T}(V_{ij}^{(i+1)})$ }	
		\State{$V_{31}^{(i+1)} = W_{31}/m_{i+1,k} $, \quad  $V_{32}^{(k+1)} = W_{32}.$}
		\State{Perform Givens rotations for the ith column of $m$:}
		\For{$ j=1,\ldots,i-1$}
		\State{$\left[ \begin{matrix}
				m_{j,i}\\
				m_{j+1,i}
			\end{matrix}\right] =
			\left[ \begin{matrix}
				c_j & s_j\\
				-s_j & c_j
			\end{matrix}\right]\left[ \begin{matrix}
				m_{j,i}\\
				m_{j+1,i}
			\end{matrix}\right]  $}
		\EndFor
		\State{Compute ith Givens rotation, and perform for $\gamma$ and last column of $m$.}
		\State{	$\left[ \begin{matrix}
				\gamma_{i}\\
				\gamma_{i+1}
			\end{matrix}\right] =
			\left[ \begin{matrix}
				c_i & s_i\\
				-s_i & c_i
			\end{matrix}\right]\left[ \begin{matrix}
				\gamma_{i}\\
				0
			\end{matrix}\right] $}
		\State{$ m_{i,i}=c_i m_{i,i}+s_i m_{i+1,i} $, \qquad  $m_{i+1,i}=0$.}
		\If{$ \lvert\gamma_{i+1}\rvert \leq \epsilon_{tol} $ }
		\State{Compute $y$ from  $My=\xi$, where $(M)_{j,i} = m_{j,i}$.}
		\State{$Y_{11} = \{y_1V_{11}^{(1)}, \ldots, y_k V_{11}^{(i)}\} $, \quad  $Y_{12} = \{V_{12}^{(1)}, \ldots, V_{12}^{(i)}\}. $}
		\State{$Y_{21} = \{y_1V_{21}^{(1)}, \ldots, y_k V_{21}^{(i)}\} $, \quad  $Y_{22} = \{V_{22}^{(1)}, \ldots, V_{22}^{(i)}\}. $\hfill $Y_{ij}\leftarrow \mathcal{T}(Y_{ij})$}
		\State{$Y_{31} = \{y_1V_{31}^{(1)}, \ldots, y_k V_{31}^{(i)}\} $, \quad  $Y_{32} = \{V_{32}^{(1)}, \ldots, V_{32}^{(i)}\}. $}
		\State{$(\widetilde{Y}_{11}, \widetilde{Y}_{12}, \widetilde{Y}_{21}, \widetilde{Y}_{22}, \widetilde{Y}_{31}, \widetilde{Y}_{32}) \hspace{-1mm} = \hspace{-1mm}\mathcal{P}^{-1}_0(Y_{11}, Y_{12}, Y_{21}, Y_{22}, Y_{31}, Y_{32}$). \hfill $\widetilde{Y}_{ij}\leftarrow \mathcal{T}(\widetilde{Y}_{ij}) $}
		\State{$\Theta_{11} = \{\Theta_{11}^{(0)}, \;\; \widetilde{Y}_{11}\} $, \quad  $\Theta_{12} = \{\Theta_{12}^{(0)}, \;\; \widetilde{Y}_{12}\}. $}
		\State{$\Theta_{21} = \{\Theta_{21}^{(0)}, \;\; \widetilde{Y}_{21}\} $, \quad  $\Theta_{22} = \{\Theta_{22}^{(0)}, \;\; \widetilde{Y}_{22}\}. $ \hfill $ \Theta_{ij}\leftarrow \mathcal{T}(\Theta_{ij})$}
		\State{$\Theta_{31} = \{\Theta_{31}^{(0)}, \;\; \widetilde{Y}_{31}\} $, \quad  $\Theta_{32} = \{\Theta_{32}^{(0)}, \;\; \widetilde{Y}_{32}\}. $}		
		\EndIf      	
		\EndWhile
	\end{algorithmic}
\end{algorithm}

During the iteration process, the rank of  low--rank factors can increase either via matrix vector products or vector (matrix) additions. Thus, the cost of rank--reduction techniques is kept under control by using truncation based on singular values \cite{DKressner_CTobler_2011} or truncation based on coarse--grid rank reduction \cite{KLee_HCElman_2017}. Our approach is based on the discussion in  \cite{MStoll_TBreiten_2015,PBenner_AOnwunto_MStoll_2015}, where  a truncated SVD of  $ U= W^TV \approx B \diag(\sigma_{1},\dots,\sigma_{r})C^T$  is constructed for the largest $r$ singular values, which are greater than the given truncation tolerance $\epsilon_{trunc}$. In Algorithm~\ref{alg:GMRES}, this process is done by the truncation operator $\mathcal{T}$. Further, in the numerical simulations,  a rather small truncation tolerance $\epsilon_{trunc}$  is used to represent the full--rank solution as accurate as possible.

We know that iterative methods such as GMRES exhibit a better convergence in terms of the number of iterations when they are used with a suitable preconditioner. The low--rank variants also display the same behaviour so that we use  a block diagonal mean-based preconditioner  of the form
\begin{align*}
	\mathcal{P}_0 =  \left[
	\begin{matrix}
		\mathcal{M}_{\gamma} & 0 & 0\\
		0 & \mu \left( \mathcal{G}_0 \otimes I \right) & 0\\
		0 & 0 & \widetilde{S}
	\end{matrix}
	\right],
\end{align*}
where $\widetilde{S} =  \big( \mathcal{G}_0 \otimes \widetilde{\mathcal{K}}_0 \big) \mathcal{M}_{\gamma}^{-1}  \big( \mathcal{G}_0 \otimes \widetilde{\mathcal{K}}_0 \big)^T$  corresponds to the approximated Schur complement with $ \widetilde{\mathcal{K}}_0 = \mathcal{K}_0 + \sqrt{\frac{1+\gamma}{\mu}} M\diag(\mathbbm{1}_{\mathcal{I}})$; see, e.g,  \cite{PBenner_AOnwunto_MStoll_2016,CEPowell_HCElman_2009}.


\section{Numerical Results}\label{sec:num}

This section contains  a set of numerical experiments to illustrate the performance of proposed discretization techniques and  a low--rank variant of GMRES approach. All numerical simulations  are done in  MATLAB R2021a on an Ubuntu Linux machine with 32 GB RAM. Iterative approaches are ended when the residual becomes smaller than the  given tolerance value $\epsilon_{tol} = 5 \times 10^{-3}$ or  the maximum iteration number ($\#iter_{max} =250$)  is reached. The truncation tolerance $\epsilon_{trunc}=10^{-8}$ is chosen, such that  $\epsilon_{trunc} \leq  \epsilon_{tol}$;  otherwise, one would iterate the noise during the low--rank process.

In the numerical experiments, the random coefficient $\eta$ is described by the following  covariance function
\begin{align}\label{Cov:Gauss}
	\mathbb{C}_{\eta} (\boldsymbol{x} ,\boldsymbol{y} ) = \kappa^2 \prod_{n=1}^{2}  e^{-\lvert x_n -y_n \rvert /\ell_n }, \quad \forall (\boldsymbol{x} ,\boldsymbol{y} ) \in \mathcal{D}
\end{align}
with the correlation length $\ell_n$.  Linear elements are used to generate discontinuous Galerkin basis, whereas Legendre polynomials are taken as the stochastic basis functions since the underlying random variables have uniform distribution  over $[-\sqrt{3},\sqrt{3}]$, that is,  $ \xi_j \sim \mathcal{U}[-\sqrt{3},\sqrt{3}], \; j=1, \ldots, N$.  Explicit eigenpairs $(\lambda_j, \phi_j)$ of the covariance function \eqref{Cov:Gauss} can be found in \cite{GJLord_CEPowell_TShardlow_2014}. Further, all  parameters used in the simulations are described in Table~\ref{tab:para}.

\begin{table}[htp!]
	\caption{Descriptions of the parameters used in the  simulations.}\label{tab:para}
	\centering
\begin{tabular}{cl}
	Parameter & Description                                             \\ \hline
	$N_d$     & degrees of freedom for the spatial discretization           \\
	$N$       & truncation number in KL expansion                       \\
	$Q$       & highest order of basis polynomials for the stochastic domain \\
	$\mu$     & regularization parameter of  the control $u$            \\
	$\gamma$  & risk-aversion parameter                                 \\
	$\nu$     & viscosity parameter                                     \\
	$\ell$    & correlation length                                      \\
	$\kappa$  & standard deviation
\end{tabular}
\end{table}

\subsection{Unconstrained problem with random diffusion parameter}\label{ex:uncos_diff}

As a first benchmark problem, we consider an unconstrained optimal control problem, that is, $\mathcal{U}^{ad}=\mathcal{U}$, having a random diffusion coefficient defined on $\mathcal{D} = [-1,1]^2$ with the source function $f(\boldsymbol{x} )=0$,  the convection parameter $ \mathbf{b}(\boldsymbol{x} )=(0,1)^T$, and the Dirichlet boundary condition
\[
y_{DB}(\boldsymbol{x})=
\begin{cases}
	y_{DB}(x_1,-1)=x_1, & y_{DB}(x_1,1)=0,\\
	y_{DB}(-1,x_2)=-1,  & y_{DB}(1,x_2)=1.
\end{cases}
\]
The random diffusion parameter is chosen as $a(\boldsymbol{x},\omega)= \nu \,\eta(\boldsymbol{x},\omega)$, where the random field $\eta(\boldsymbol{x},\omega)$ has the unity mean with the corresponding covariance function \eqref{Cov:Gauss} and $\nu$ is the viscosity parameter. The desired state (or target) $y^d$ corresponds to the stochastic solution of the forward model by taking $u(\boldsymbol{x})=0$. The desired state exhibits exponential boundary layer near $x_2 =1$, where  the solution changes in a dramatic manner. Therefore, the boundary layer becomes more visible as $\nu$ decreases; see Figure~\ref{fig:ex1_solns} for the mean of the state $\mathbb{E}[y_h]$, the desired state $y^d$,  and the corresponding control $u_h$ for various values of the viscosity parameter $\nu$.

\begin{figure}[htp!]
	\centering
	\includegraphics[width=1.0\textwidth]{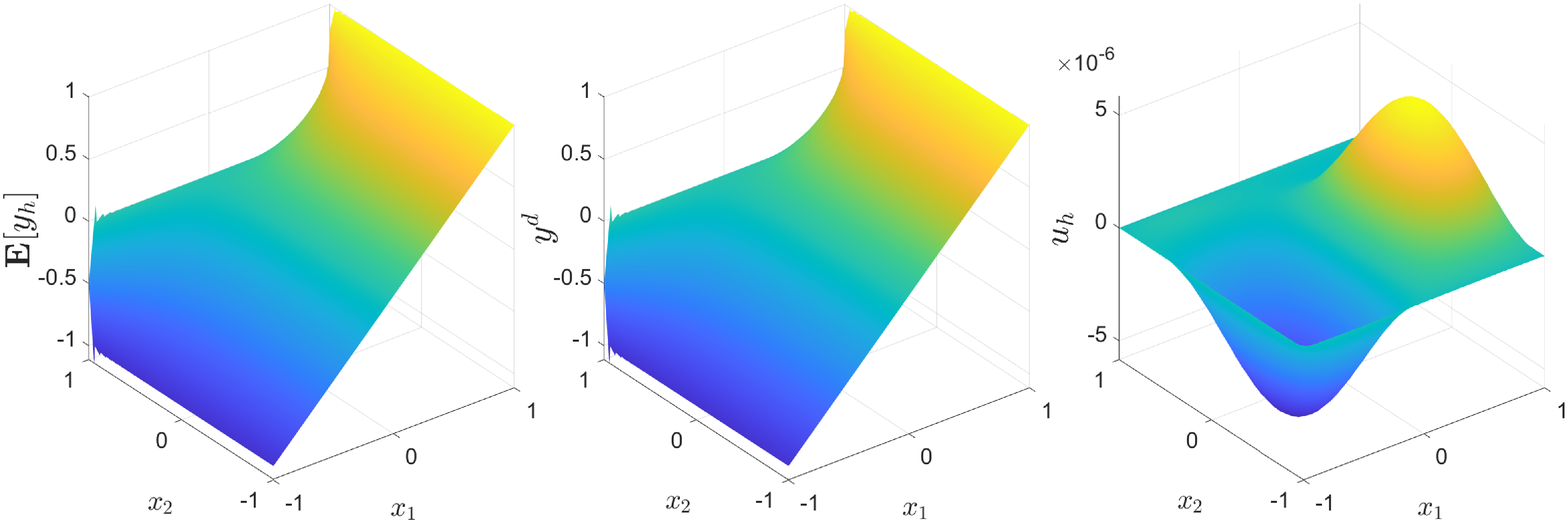}	
	\includegraphics[width=1.0\textwidth]{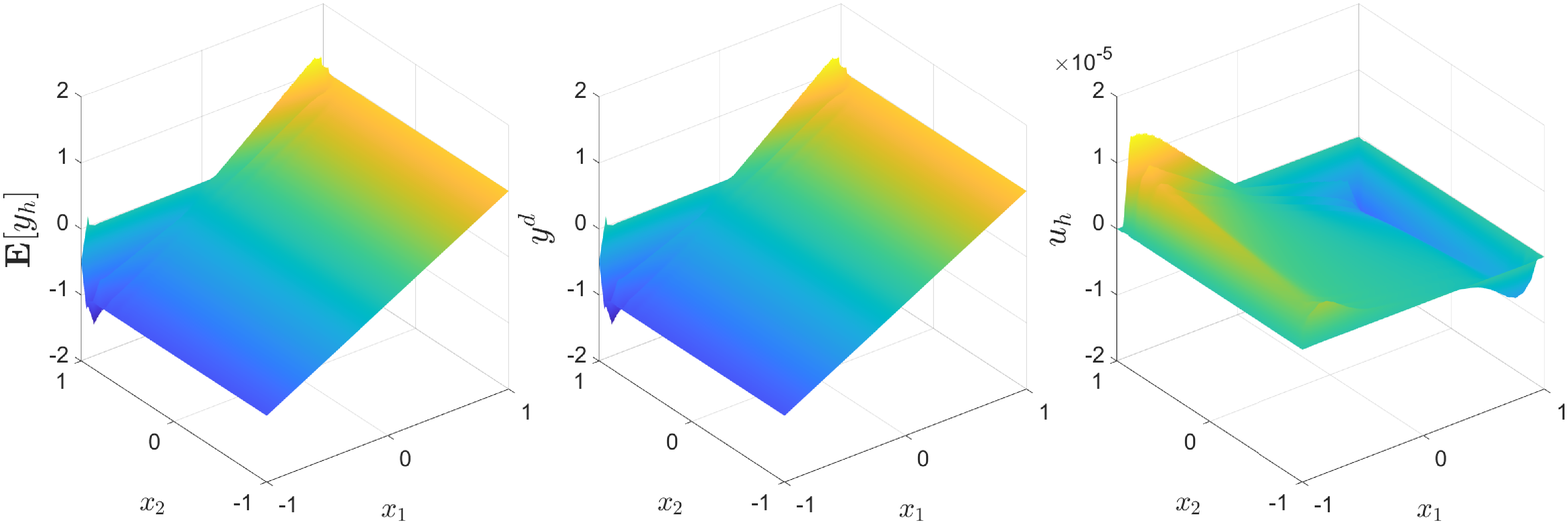}	
	\includegraphics[width=1.0\textwidth]{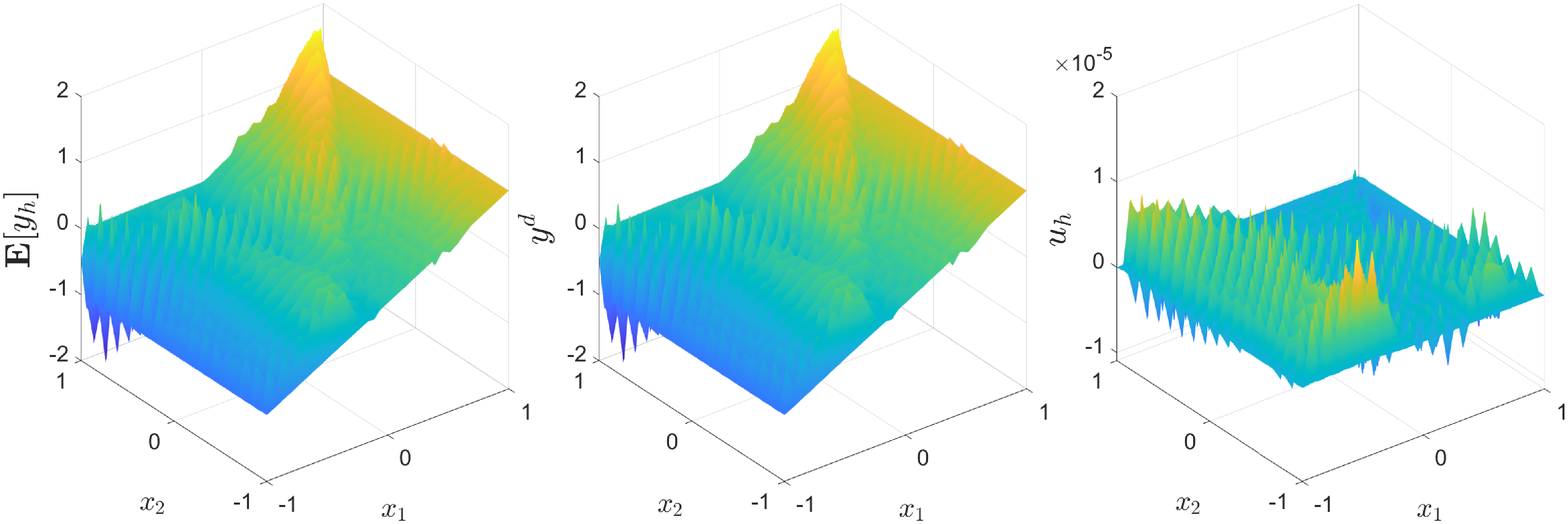}	
	\caption{Example~\ref{ex:uncos_diff}:  Simulations of  the mean of state $\mathbb{E}[y_h]$, the desired state (target) $y^d$, and the control $u_h$ (from left to right) obtained by $\mathcal{L}\backslash \mathcal{B}$ with $N_d=6144$, $N=3$, $Q=3$, $\ell=1$, $\kappa=0.05$, $\mu = 1$, and $\gamma = 1$  for varying $\nu=1, 10^{-2}, 10^{-4}$ (from top to bottom).}
	\label{fig:ex1_solns}	
\end{figure}

\begin{table}[htp!]
	\caption{Example~\ref{ex:uncos_diff}: Computational values of the cost functional $\mathcal{J}(u_h)$ and tracking term $\lVert y_h-y^d\rVert^2_{\mathcal{X}}$ obtained by $\mathcal{L}\backslash \mathcal{B}$ with $N_d=6144$, $N=3$, $Q=3$, $\ell=1$, $\kappa=0.5$, and $\gamma=1$ for varying values of the viscosity parameter $\nu$ and the regularization parameter $\mu$.}\label{tab:ex1_numu}
	\centering
	\begin{tabular}{cccccc}
		&  & $\mu = 1$  & $\mu = 10^{-2}$ & $\mu = 10^{-4}$ & $\mu = 10^{-6}$ \\
		\hline
		\hline
		\multicolumn{1}{c|}{\multirow{2}{*}{$\nu=1$}}
		& \multicolumn{1}{c|}{$\mathcal{J}(u_h)$} & 1.2393e-05 & 2.5034e-06 &1.0257e-06  & 9.7533e-07 \\
		\multicolumn{1}{c|}{} & \multicolumn{1}{c|}{$ \lVert y_h-y^d\rVert^2_{\mathcal{X}}$}
		& 5.4679e-06 & 7.1725e-07  & 3.2113e-08 & 1.6931e-09\\
		\hline
		\hline
		\multicolumn{1}{c|}{\multirow{2}{*}{$\nu=10^{-2}$}}
		& \multicolumn{1}{c|}{$\mathcal{J}(u_h)$} & 1.4349e-05 & 8.3285e-06& 7.2067e-07& 6.0683e-07 \\
		\multicolumn{1}{c|}{} & \multicolumn{1}{c|}{$ \lVert y_h-y^d\rVert^2_{\mathcal{X}}$}
		& 1.3120e-05 & 4.1452e-06  & 9.0731e-08  & 3.5983e-09 \\
		\hline
		\hline
		\multicolumn{1}{c|}{\multirow{2}{*}{$\nu=10^{-4}$}} & \multicolumn{1}{c|}{$\mathcal{J}(u_h)$}
		& 1.3675e-05 & 1.1798e-06 & 3.9380e-07 & 3.7211e-07 \\
		\multicolumn{1}{c|}{} & \multicolumn{1}{c|}{$ \lVert y_h-y^d\rVert^2_{\mathcal{X}}$}
		& 1.5285e-05 & 4.3924e-07 & 1.1422e-07 & 8.3896e-09 \\
		\hline
		\hline
	\end{tabular}
\end{table}

\begin{table}[htp!]
	\caption{Example~\ref{ex:uncos_diff}: Peak values of the states' variance obtained by $\mathcal{L}\backslash \mathcal{B}$ with $N_d=6144 $, $N=3$, $Q=3$, $\ell=1$,  $\nu=1$, and $\mu = 1$  for varying values of the risk--aversion $\gamma$ and the standard deviation $\kappa$.}\label{tab:ex1_variance}
	\centering
	\begin{tabular}{cccc}
		\multicolumn{1}{l}{}                    & \multicolumn{1}{l}{$\kappa = 0.05$} & \multicolumn{1}{l}{$\kappa = 0.25$} & \multicolumn{1}{l}{$\kappa = 0.5$} \\ \hline
		\multicolumn{1}{c|}{$\gamma = 0$}       & 4.5406e-05                          & 1.1980e-03                          & 5.7995e-03                          \\
		\multicolumn{1}{c|}{$\gamma = 1$}       & 4.1995e-05                          & 1.0984e-03                          & 5.1327e-03                          \\
		\multicolumn{1}{c|}{$\gamma = 2$}       & 3.8944e-05                          & 1.0110e-03                          & 4.5807e-03                          \\
		\multicolumn{1}{c|}{$\gamma = 3$}       & 3.6207e-05                          & 9.3377e-04                          & 4.1243e-03                          \\
		\multicolumn{1}{c|}{$\gamma = 4$}       & 3.3731e-05                          & 8.6520e-04                          & 3.7409e-03
	\end{tabular}
\end{table}

\begin{table}[htp!]
	\caption{Example~\ref{ex:uncos_diff}: Total CPU times (in seconds) and memory (in KB) for $N_d=6144 $, $Q=3$, $\ell=1$, $\mu = 10^{-2}$, $\gamma = 1$, and $\kappa=0.5$.}
	\label{tab:ex1_momory}
	\centering
	\begin{tabular}{clll}
		$\mathcal{L}\backslash \mathcal{B}$ & $\nu=10^{0}$  & $\nu=10^{-2}$  & $\nu=10^{-4}$  \\ \hline
		N                                   & CPU (Memory)  & CPU (Memory)   & CPU (Memory)   \\ \hline
		2                                   & 116.0 (2880)  & 116.1 (2880)   & 117.5 (960)    \\
		3                                   & 779.6 (5760)  & 787.7 (5760)   & 813.0 (5760) \\
		4                                   & OoM           & OoM            & OoM
	\end{tabular}
\end{table}

\begin{table}[htp!]
	\caption{Example~\ref{ex:uncos_diff}: Total number of iterations,  total rank of the truncated solutions, total CPU times (in seconds),  relative residual, and memory demand of the solution (in KB) with $N_d=6144 $, $Q=3$, $\ell=1$, $\kappa=0.5$,  $\nu=1$,  $\gamma=0$, and  the mean-based preconditioner $\mathcal{P}_0$ for varying values of $N$ and $\mu$.} \label{tab:kappa2}
	\centering
	\begin{tabular}{ccccc}
		&                             & $\mu =1$     & $\mu= 10^{-2}$ & $\mu = 10^{-4}$ \\
		\hline
		\hline
		\multicolumn{1}{c|}{\multirow{5}{*}{$N = 4$}}
		& \multicolumn{1}{c|}{\#iter} & 250 & 250 & 250  \\
		\multicolumn{1}{c|}{}
		& \multicolumn{1}{c|}{Rank}   & 51 & 51 &  51 \\
		\multicolumn{1}{c|}{}
		& \multicolumn{1}{c|}{CPU}    & 40126.1 & 40017.9 & 39950.4  \\
		\multicolumn{1}{c|}{}
		& \multicolumn{1}{c|}{Resi.}  & 3.5759e-02 & 3.3521e-02 & 3.7375e-02  \\
		\multicolumn{1}{c|}{}
		& \multicolumn{1}{c|}{Memory} & 2461.9 & 2461.9 & 2461.9  \\
		\hline
		\hline
		\multicolumn{1}{c|}{\multirow{5}{*}{$N = 5$}}
		& \multicolumn{1}{c|}{\#iter} & 250 & 250 & 250  \\
		\multicolumn{1}{c|}{}
		& \multicolumn{1}{c|}{Rank}   & 84 & 84 & 84  \\
		\multicolumn{1}{c|}{}
		& \multicolumn{1}{c|}{CPU}    & 91366.2 & 90544.0 & 90021.2  \\
		\multicolumn{1}{c|}{}
		& \multicolumn{1}{c|}{Resi.}  & 2.1672e-02 & 2.3056e-02 & 3.1080e-02  \\
		\multicolumn{1}{c|}{}
		& \multicolumn{1}{c|}{Memory} & 4068.8 & 4068.8 & 4068.8  \\
		\hline
		\hline
		\multicolumn{1}{c|}{\multirow{5}{*}{$N = 6$}}
		& \multicolumn{1}{c|}{\#iter} & 250 & 250 & 250  \\
		\multicolumn{1}{c|}{}
		& \multicolumn{1}{c|}{Rank}   & 126 & 126 & 126  \\
		\multicolumn{1}{c|}{}
		& \multicolumn{1}{c|}{CPU}    & 208643.4 & 207964.0 & 207464.4  \\
		\multicolumn{1}{c|}{}
		& \multicolumn{1}{c|}{Resi.}  & 1.8357e-02 & 1.8064e-02 & 2.0494e-02 \\
		\multicolumn{1}{c|}{}
		& \multicolumn{1}{c|}{Memory} & 6130.7 & 6130.7 & 6130.7 \\
		\hline
		\hline
		\multicolumn{1}{c|}{\multirow{5}{*}{$N = 7$}}
		& \multicolumn{1}{c|}{\#iter} & 250 & 250 & 250  \\
		\multicolumn{1}{c|}{}
		& \multicolumn{1}{c|}{Rank}   &  180 &  180 & 180  \\
		\multicolumn{1}{c|}{}
		& \multicolumn{1}{c|}{CPU}    & 355115.9 & 355167.5 & 355652.3  \\
		\multicolumn{1}{c|}{}
		& \multicolumn{1}{c|}{Resi.}  & 1.1208e-02 & 1.3833e-02 & 1.4914e-02 \\
		\multicolumn{1}{c|}{}
		& \multicolumn{1}{c|}{Memory} & 8808.8 & 8808.8 & 8808.8  \\
		\hline
		\hline
	\end{tabular}
\end{table}

Table~\ref{tab:ex1_numu} shows the values of the cost functional $\mathcal{J}(u_h)$ and tracking term $\lVert y_h-y^d\rVert^2_{\mathcal{X}}$ obtained by $\mathcal{L}\backslash \mathcal{B}$ for  various values of the viscosity parameter $\nu$ and the regularization parameter $\mu$. We observe that  the tracking term   and the objective functional  become smaller as $\mu$ decreases. Moreover, Table~\ref{tab:ex1_variance} exhibits that  the peak values of states' variance  can be reduced by increasing the value of the
parameter $\gamma$.

Next, we display the performance of $\mathcal{L}\backslash \mathcal{B}$ in terms of total CPU times (in seconds) and storage requirements (in KB) in Table~\ref{tab:ex1_momory}. However, we could not report some numerical results since the simulation is ended with \textquotedblleft out of memory\textquotedblright, which we have denoted as \textquotedblleft OoM\textquotedblright. To handle the curse of dimensionality and so increase the value of truncation number $N$, we need  effective numerical approaches or solvers such as a low--rank variant of GMRES iteration with a mean based preconditioner discussed in Section~\ref{sec:lowrank}.

\begin{figure}[htp!]
	\centering
	\includegraphics[width=1.0\textwidth]{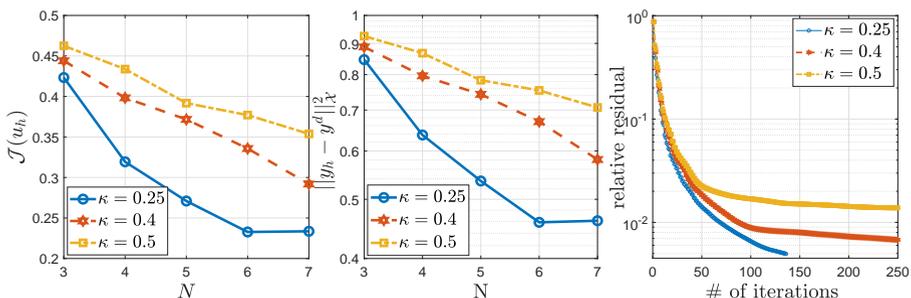}	
	\caption{Example~\ref{ex:uncos_diff}: Behaviours of the cost functional $\mathcal{J}(u_h)$ (left),  the tracking term $\lVert y_h-y^d\rVert^2_{\mathcal{X}}$ (middle), and the relative residual (right) with $N_d=6144 $, $Q=3$, $\ell=1$,  $\nu=1$, $\mu = 10^{-2}$, $\gamma=0$, and  the mean-based preconditioner $\mathcal{P}_0$ for varying values of $\kappa$.}
	\label{fig:Ex1_kappa}		
\end{figure}

Table~\ref{tab:kappa2}  reports the results of the simulations by considering various data sets in the low--rank format. By keeping other parameters fixed, we show results for varying truncation number $N$ in KL expansion  and regularization parameter $\mu$  for $\kappa=0.5$ in  Table~\ref{tab:kappa2}.  When $N$ increases, the complexity of the problem increases in terms of the number of rank,  memory, and CPU time. Another key observation is that  the relative residual decreases independently  of the value of $\mu$ while increasing $N$.

Next, we investigate the effect of the standard deviation parameter $\kappa$ on the numerical simulations. Figure~\ref{fig:Ex1_kappa} displays the behaviours of the cost functional $\mathcal{J}(u_h)$, the tracking term $\lVert y_h-y^d\rVert^2_{\mathcal{X}}$, and the relative residual for various values of $\kappa$. We observe that the values of $\mathcal{J}(u_h)$ and $\lVert y_h-y^d\rVert^2_{\mathcal{X}}$ decrease monotonically as the value of $\kappa$ increases. Moreover, the  low--rank variant of preconditioned GMRES method yields convergence behaviour for all values of $\kappa$. Lastly, Figure~\ref{fig:Ex1_resgamma} shows that the speed of convergence of relative residual decreases  by increasing the value of risk--aversion parameter $\gamma$ in the beginning of the iteration.

\begin{figure}[htp!]
	\centering
	\includegraphics[width=1.0\textwidth]{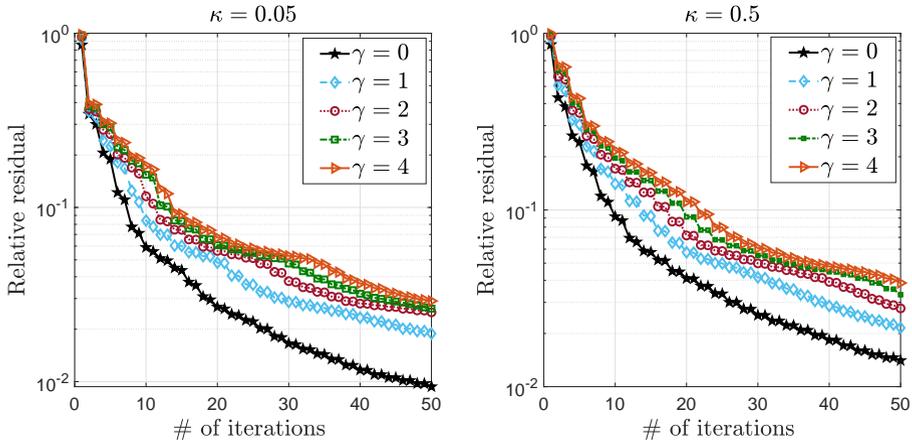}	
	\caption{Example~\ref{ex:uncos_diff}: Convergence of LRPGMRES with $N_d=6144 $, $N=5$, $Q=3$, $\ell=1$, $\mu=1$, $\nu=1$ for varying $\kappa$ and $\gamma$.}
	\label{fig:Ex1_resgamma}		
\end{figure}

\subsection{Unconstrained problem with random convection parameter} \label{ex:uncos_conv}

Our second example is an unconstrained optimal control problem  containing random velocity input parameter. To be precise, we set the deterministic diffusion parameter $a(\boldsymbol{x},\omega)=\nu > 0$, the deterministic source function $f(\boldsymbol{x})=0$, and homogeneous Dirichlet boundary conditions on the spatial domain $\mathcal{D} = [-1,1]^2$. On the other hand, the random velocity field $\mathbf{b}(\boldsymbol{x},\omega)$ is defined as
$
\mathbf{b}(\boldsymbol{x},\omega)= \left( \eta(\boldsymbol{x},\omega), \eta(\boldsymbol{x},\omega) \right)^T,
$
where the random input $\eta(\boldsymbol{x},\omega)$  has the unity mean, i.e., $\overline{\eta}(\boldsymbol{x}) = 1$. Further, the desired state $y^d$ is given by
\[
y^d(\boldsymbol{x}) = \exp \left[ -64 \left( \left( x_1- \frac{1}{2}\right)^2 + \left( x_2 - \frac{1}{2}\right)^2 \right)\right].
\]

Figure~\ref{fig:Ex2_statemean} and ~\ref{fig:Ex2_controlmean}	 display, respectively,  the  mean of state $\mathbb{E}[y_h]$ and the  control $u_h$ for varied values of the regularization parameter $\mu$ obtained by solving the full--rank system $\mathcal{L}\backslash \mathcal{B}$. As the previous example, we observe that  the state $y_h$ becomes closer to the target solution $y^d$ while $\mu$ decreases.

\begin{figure}[htp!]
	\centering
	\includegraphics[width=1\textwidth]{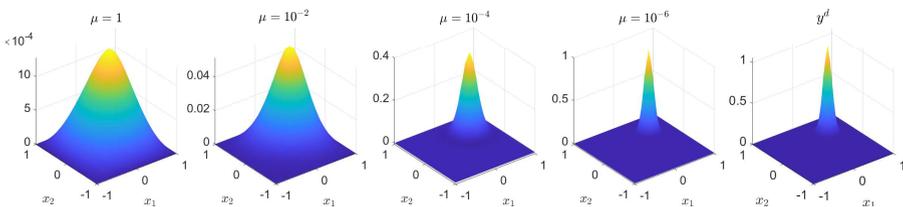}	
	\caption{Example~\ref{ex:uncos_conv}: Simulations of the mean  of state $\mathbb{E}[y_h]$ obtained by  $\mathcal{L}\backslash \mathcal{B}$  with $N_d=6144$, $N=3$, $Q=3$, $\ell=1$, $\kappa=0.05$, $\nu=1$, and $\gamma = 0$  for varying  $\mu=1, 10^{-2}, 10^{-4}, 10^{-6}$ and the desired state $y^d$.}
	\label{fig:Ex2_statemean}		
\end{figure}

\begin{figure}[htp!]
	\centering
	\includegraphics[width=1\textwidth]{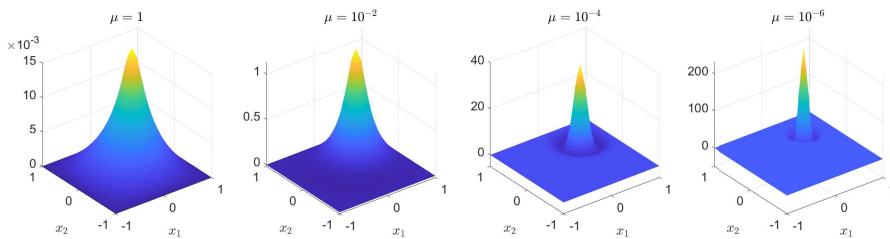}		
	\caption{Example~\ref{ex:uncos_conv}: Simulations  of the control $u_h$ obtained solving by $\mathcal{L} \backslash \mathcal{B}$ with $N_d=6144$, $N=3$, $Q=3$, $\ell=1$, $\kappa=0.05$,  $\nu=1$, and $\gamma = 0$  for varying regularization parameter $\mu=1, 10^{-2}, 10^{-4}, 10^{-6}$.}
	\label{fig:Ex2_controlmean}	
\end{figure}

\begin{figure}[htp!]
	\centering
	\includegraphics[width=1\textwidth]{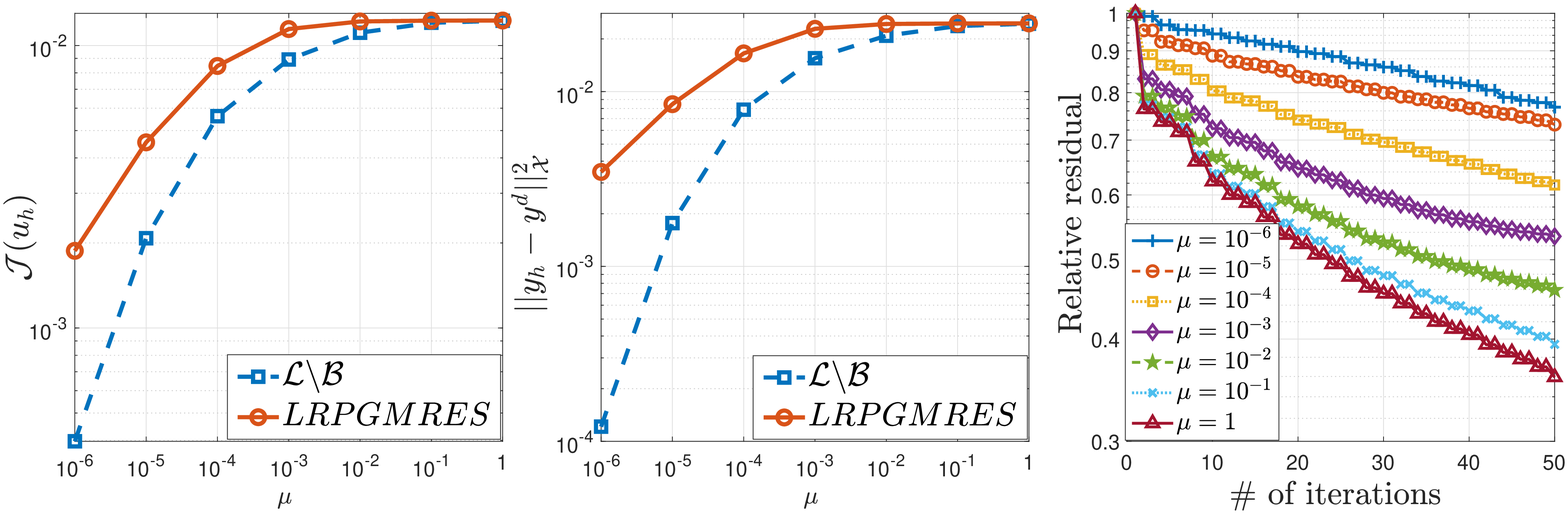}		
	\caption{Example~\ref{ex:uncos_conv}: Behaviours of the cost functional $\mathcal{J}(u_h)$ (left),  the tracking term $\lVert y_h-y^d\rVert^2_{\mathcal{X}}$ (middle), and the relative residual (right) with $N_d=6144 $, $N=3$, $Q=3$, $\kappa = 0.05$, $\ell=1$,  $\nu=1$,  $\gamma=0$, and  the mean-based preconditioner $\mathcal{P}_0$ for varying $\mu$.}
	\label{fig:Ex2_mu}		
\end{figure}

\begin{table}[htp!]
	\caption{Example~\ref{ex:uncos_conv}: Simulation results showing total number of iterations, ranks of the truncated solutions, total CPU times (in seconds),  relative residual, and memory demand of the solution (in KB) with $N_d=6144 $, $N=3$, $Q=3$, $\ell=1$,  $\nu=1$, $\mu = 10^{-6}$, and  the mean-based preconditioner $\mathcal{P}_0$ for varying  $\gamma$.} \label{tab:Ex2_gamma}
	\centering
	\begin{tabular}{cccccc}
		& $\gamma =0$ & $\gamma= 10^{-6}$ & $\gamma= 10^{-4}$ & $\gamma= 10^{-2}$ & $\gamma= 1$  \\ \hline
		\multicolumn{1}{c|}{\#iter} & 250         & 250               & 250               & 250               & 250                \\
		\multicolumn{1}{c|}{Rank}   & 29          & 30                & 30                & 30                & 21                  \\
		\multicolumn{1}{c|}{CPU}    & 24468.2     & 19383.2           & 17382.0           & 17422.8           & 17797.1         \\
		\multicolumn{1}{c|}{Resi.}  & 2.1733e-01  & 2.6663e-01        & 4.0428e-01        & 6.9542e-01        & 9.1911e-01   \\
		\multicolumn{1}{c|}{Memory} & 1396.5      & 1444.7            & 1444.7            & 1444.7            & 963.2
	\end{tabular}
\end{table}

Next, we compare the full--rank solutions obtained by solving the  system $\mathcal{L}\backslash \mathcal{B}$ with the low--rank ones. Figure~\ref{fig:Ex2_mu} exhibits  behaviours of the cost functionals $\mathcal{J}(u_h)$ (left),  the tracking term $\lVert y_h-y^d\rVert^2_{\mathcal{X}}$ (middle), and the relative residual (right) for varying values of the regularization parameter $\mu$. The key observation is that the low--rank solutions display the same pattern with the full--rank solutions as $\mu$ increases. Moreover, Table~\ref{tab:Ex2_gamma}  reports the results of the simulations by considering various values of the  risk--aversion parameter $\gamma$. As the previous example, the relative residual becomes smaller as decreasing the value of $\gamma$.

\begin{figure}[htp!]
	\centering
	\includegraphics[width=1\textwidth]{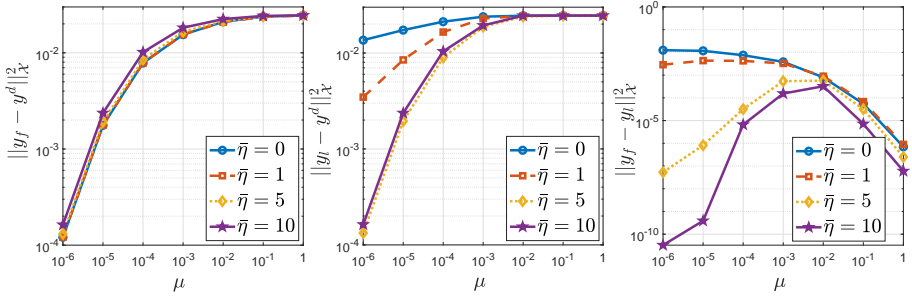}		
	\caption{Example~\ref{ex:uncos_conv}: Behaviour of the differences  $\lVert y_f-y^d\rVert^2_{\mathcal{X}}$ (left), $\lVert y_l-y^d\rVert^2_{\mathcal{X}}$ (middle), and $\lVert y_f-y^l\rVert^2_{\mathcal{X}}$ (right), where  the full--rank and low--rank solutions are denoted by $y_f$ and $y_l$, respectively, computed by solving the full--rank  and low--rank systems  with $N_d=6144 $, $N=3$, $Q=3$, $\ell=1$, $\mu=10^{-6}$, $\gamma = 0$, $\nu=1$, and $\kappa=0.05$ for varying  values of the  mean of random input $\eta(x)$.} 
	\label{fig:Ex2_meanparameter}	
\end{figure}

Last, we investigate the effect of the  mean of random input $\eta(\boldsymbol{x} )$ on both full--rank and low--rank solutions. Denoting the full--rank solution and the low-rank solution by $y_f$ and $y_l$, respectively,  the behavior of  the differences  $\lVert y_f-y^d\rVert^2_{\mathcal{X}}$, $\lVert y_l-y^d\rVert^2_{\mathcal{X}}$, and $\lVert y_f-y^l\rVert^2_{\mathcal{X}}$ computed by solving the full--rank and low--rank systems is displayed in Figure~\ref{fig:Ex2_meanparameter}.  As increasing the mean  of random input $\eta(\boldsymbol{x} )$, the difference between the full--rank and low--rank solutions becomes smaller.


\begin{figure}[htp!]
	\centering
	\includegraphics[width=1\textwidth]{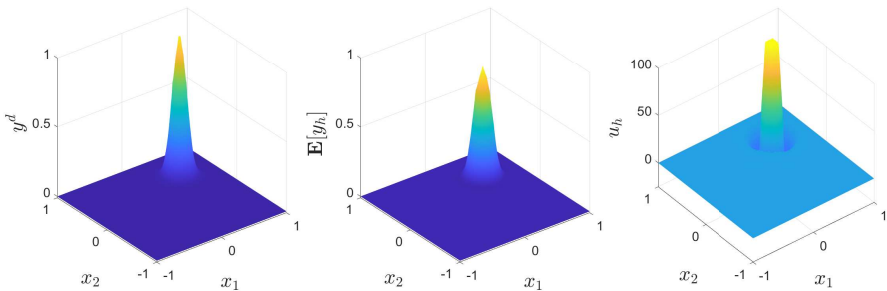}	
	\caption{Example~\ref{ex:cos_conv}: Simulations of the desired state $y^d$, the mean of state $\mathbb{E}[y_h]$,  and the control $u_h$ (from left to right) obtained  by $\mathcal{L}\backslash \mathcal{B}$ with $N_d=6144$, $N=3$, $Q=3$, $\ell=1$, $\kappa=0.05$, and $\nu=1$.}
	\label{fig:Ex3_statemean}		
\end{figure}

\begin{table}[htp!]
	\centering
	\caption{Example~\ref{ex:cos_conv}:  Simulation results showing the memory demand of the solution (in KB), the objective function $\mathcal{J}(u_h)$, the tracking term $\lVert y_h-y^d\rVert^2_{\mathcal{X}}$,  the difference of the full--rank and low--rank $\lVert y_{f}-y_{l}\rVert^2_{\mathcal{X}}$, ranks of the truncated solutions,  and the relative residual with $N_d=6144 $, $Q=3$, $\ell=1$,  $\nu=1$, and the mean-based preconditioner $\mathcal{P}_0$.} \label{tab:Ex3}
	\begin{tabular}{ccccccc}
		& Memory & $\mathcal{J}(u_h)$ & $\lVert y_h-y^d\rVert^2_{\mathcal{X}}$ & $\lVert y_{f}-y_{l}\rVert^2_{\mathcal{X}}$  & Rank & Res.       \\ \hline
		\multicolumn{1}{c|}{$N=3$}  & 5744.0 & 5.508e-04     & 6.031e-04                            &                                             &      &            \\ \hline
		\multicolumn{1}{c|}{$N=3$}  & 1444.7 & 1.046e-02     & 2.091e-02                            & 1.802e-02                                   & 30   & 9.232e-01  \\
		\multicolumn{1}{c|}{$N=4$}  & 2461.9 & 1.029e-02     & 2.056e-02                            & 1.769e-02                                   & 51   & 9.161e-01  \\
		\multicolumn{1}{c|}{$N=5$}  & 4068.8 & 9.996e-03     & 1.996e-02                            & 1.713e-02                                   & 84   & 9.042e-01  \\
		\multicolumn{1}{c|}{$N=6$}  & 6130.7 & 9.616e-03     & 1.919e-02                            & 1.642e-02                                    & 126   & 8.895e-01
	\end{tabular}
\end{table}

\subsection{Constrained problem with random convection parameter} \label{ex:cos_conv}
Last, we consider a constrained optimal control problem containing a random velocity parameter. Except from the set up of  Example~\ref{ex:uncos_conv}, we have an upper bound for the control variable such as $u_b =100$. Taking the results in the previous example into account, the regularization and risk--averse parameters are chosen as $\mu = 10^{-6}$ and $\gamma = 0$, respectively.

Figure~\ref{fig:Ex3_statemean} displays the desired state $y^d$, the mean of state $\mathbb{E}[y_h]$,  and the control $u_h$  obtained  by $\mathcal{L}\backslash \mathcal{B}$. We observe that the upper bound of the control constrained is satisfied. In Table~\ref{tab:Ex3}, we compare the low--rank solutions with the full--rank ones. As increasing the truncation number $N$, we obtain better results as expected.


\section{Conclusions}\label{sec:conc}

In this paper, we have numerically studied the statistical moments of a robust deterministic optimal control problem subject to a convection diffusion equation having random coefficients. With the help of the stochastic discontinuous Galerkin method, we transform the original problem into a large system consisting of deterministic optimal control problems for each realization of the random coefficients. However, we could not obtain some numerical results when increasing the value of truncation number $N$. Therefore, to reduce computational time and memory requirements, we have used a low--rank variant of  GMRES iteration with a mean based preconditioner (LRPGMRES). It has been shown in the numerical simulations that LRPGMRES can be an alternative to solve such large systems. As a future study, randomness can be considered in different forms, for instance, in boundary conditions, desired state, or geometry. Moreover, to handle curse of dimensionality, reduced order models, see, e.g., \cite{MGunzburger_JMing_2011,PChen_AQuarteroni_2014}, can be an alternative to the low--rank approximations.


\bmhead{Funding}
This work was supported by TUBITAK 1001 Scientific and Technological Research Projects Funding Program with project number 119F022. The authors would also like to express their sincere thanks to the anonymous referees for their most valuable suggestions.

\section*{Declarations}

\bmhead{Conflict of interest}
The authors declare no competing interests.






\end{document}